\documentclass{amsart}  
\usepackage{amssymb, tikz, float, enumitem, mathtools}
\usepackage[all, cmtip]{xy}

\usetikzlibrary{calc,decorations.markings,matrix,arrows,positioning}
\tikzset{>=latex}
\usepackage[final]{showkeys}       
\usepackage[breaklinks, final]{hyperref}

\newcommand{\A}{\mathbb{A}}

\theoremstyle{plain} \newtheorem{thm}{Theorem}[section]
\newtheorem{prop}[thm]{Proposition}
\newtheorem{lem}[thm]{Lemma}

\theoremstyle{definition} \newtheorem{defn}[thm]{Definition}
\theoremstyle{remark} 
\theoremstyle{plain}

\DeclareMathOperator{\Cg}{Cg}
\DeclareMathOperator{\Sg}{Sg}
\DeclareMathOperator{\Con}{Con}

\numberwithin{equation}{section}  
\renewcommand{\phi}{\varphi}
\renewcommand{\epsilon}{\varepsilon}

\usepackage{caption, subcaption, xparse}

\theoremstyle{definition} 
\theoremstyle{remark} 

\DeclareMathOperator{\Pol}{Pol}

\DeclareMathOperator{\Mat}{Mat}

\DeclareMathOperator{\face}{face}
\DeclareMathOperator{\cube}{cube}
\DeclareMathOperator{\crsec}{crsec}
\DeclareMathOperator{\Line}{line}
\DeclareMathOperator{\Square}{square}
\DeclareMathOperator{\Corner}{Corner}
\DeclareMathOperator{\Corners}{Corners}

\def\nat{\mathbb{N}}
\def\A{\mathbb{A}}
\def\var{\mathcal{V}}

\def\Meet{\bigwedge}
\def\Union{\bigcup}
\def\meet{\wedge}
\def\join{\vee}

\def\union{\cup}

\tikzset{myStyle/.style={baseline=(center.base), font=\small,
    every node/.style={inner sep=0.25em} }}

\NewDocumentCommand{\LinePic}{ O{} O{} O{1} }{ 
  \begin{tikzpicture}[myStyle, scale=#3*1 ]
    \node (center) at (0,0.5) {\phantom{$\cdot$}}; 
    \path (0,0)  node (s) {$#1$}
        ++(0,1)  node (n) {$#2$};
    \draw (n) -- (s);
  \end{tikzpicture}
}  

\NewDocumentCommand{\SQuare}{ O{} O{} O{} O{} O{1} }{ 
  \begin{tikzpicture}[myStyle, scale=#5*1 ]
    \node (center) at (0.5,-0.5) {\phantom{$\cdot$}}; 
    \path (0,0)  node (nw) {$#2$}
        ++(1,0)  node (ne) {$#4$}
        ++(0,-1) node (se) {$#3$}
        ++(-1,0) node (sw) {$#1$};
    \draw (nw) -- (ne) -- (se) -- (sw) -- (nw);
  \end{tikzpicture}
}  

\NewDocumentCommand{\DeltaZeroSQuare}{ O{} O{} O{} O{} O{1} }{ 
  \begin{tikzpicture}[myStyle, scale=#5*1 ]
    \node (center) at (0.5,-0.5) {\phantom{$\cdot$}}; 
    \path (0,0)  node (nw) {$#2$}
        ++(1,0)  node (ne) {$#4$}
        ++(0,-1) node (se) {$#3$}
        ++(-1,0) node (sw) {$#1$};
    \draw (nw) -- (ne) -- (se) -- (sw) -- (nw);
    \draw (nw) edge[bend right] node[left] {$\delta$}(sw);
    \draw (ne) edge[dotted, bend left]  (se);
  \end{tikzpicture}
}  

\NewDocumentCommand{\DeltaOneSQuare}{ O{} O{} O{} O{} O{1} }{ 
  \begin{tikzpicture}[myStyle, scale=#5*1 ]
    \node (center) at (0.5,-0.5) {\phantom{$\cdot$}}; 
    \path (0,0)  node (nw) {$#2$}
        ++(1,0)  node (ne) {$#4$}
        ++(0,-1) node (se) {$#3$}
        ++(-1,0) node (sw) {$#1$};
    \draw (nw) -- (ne) -- (se) -- (sw) -- (nw);
    \draw (nw) edge[dotted, bend left] (ne);
    \draw (sw) edge[ bend right] node[below] {$\delta$} (se);
  \end{tikzpicture}
}  

\NewDocumentCommand{\Axes}{ O{} O{} O{} O{1} O{0} O{1} O{2} }{  
  \begin{tikzpicture}[myStyle, scale=#4*0.85] 
    \draw (0,0) -- node[above]{$#1$} (1,0) node[right]{$#5$}
      (0,0) -- node[left]{$#2$} (0,1) node[above]{$#6$}
      (0,0) -- node[below left=-0.25em]{$#3$} (0.5,-0.5) node[below right=-0.2em]{$#7$};
    \node (center) at (0.5,0.75) {\phantom{$\cdot$}};
  \end{tikzpicture}
}  

\newcommand{\CubeNodes}[8]{  
  \node at (0.75,-0.75) (center) {\phantom{$\cdot$}}; 
  \path (0,0)  node (back_nw)      {$#2$}
      ++(1,0)  node (back_ne)      {$#4$}
      ++(0,-1) node (back_se)      {$#3$}
      ++(-1,0) node (back_sw)      {$#1$}
        (0.5,-0.5) node (front_nw) {$#6$}
      ++(1,0)  node (front_ne)     {$#8$}
      ++(0,-1) node (front_se)     {$#7$}
      ++(-1,0) node (front_sw)     {$#5$};
}  
\newcommand{\CubeUnwrapped}[8]{ 
  \CubeNodes{#1}{#2}{#3}{#4}{#5}{#6}{#7}{#8}
  \draw (back_nw) -- (back_ne) -- (back_se) -- (back_sw) -- (back_nw)
    (front_nw) -- (front_ne) -- (front_se) -- (front_sw) -- (front_nw)
    (back_nw) -- (front_nw)
    (back_ne) -- (front_ne)
    (back_se) -- (front_se)
    (back_sw) -- (front_sw);
}  
\newcommand{\CubeDUnwrapped}[8]{ 
  \CubeNodes{#1}{#2}{#3}{#4}{#5}{#6}{#7}{#8}
  \draw (front_nw) -- (front_ne) -- (front_se) -- (front_sw) -- (front_nw)
    (back_nw) -- (back_ne) (back_sw) -- (back_nw)
    (back_nw) -- (front_nw)
    (back_ne) -- (front_ne)
    (back_sw) -- (front_sw);
  \draw[densely dotted] (back_ne) -- (back_se) -- (back_sw)
    (back_se) -- (front_se);
}  
\NewDocumentCommand{\Cube}{ O{} O{} O{} O{} O{} O{} O{} O{} O{1} }{  
  \begin{tikzpicture}[myStyle, scale=#9*1 ]
    \CubeUnwrapped{#1}{#2}{#3}{#4}{#5}{#6}{#7}{#8}
  \end{tikzpicture}
}  
\NewDocumentCommand{\CubeD}{ O{} O{} O{} O{} O{} O{} O{} O{} O{1} }{  
  \begin{tikzpicture}[myStyle, scale=#9*1 ]
    \CubeDUnwrapped{#1}{#2}{#3}{#4}{#5}{#6}{#7}{#8}
  \end{tikzpicture}
}  

\NewDocumentCommand{\DeltaZeroCubeD}{ O{} O{} O{} O{} O{} O{} O{} O{} O{1} }{  
  \begin{tikzpicture}[myStyle, scale=#9*1]
    \CubeDUnwrapped{#1}{#2}{#3}{#4}{#5}{#6}{#7}{#8}
    \draw (back_sw)  to[out=30,in=180-30] (back_se)
      (back_nw)  to[out=30,in=180-30] node[auto]{$\delta$} (back_ne)
      (front_sw) to[out=30,in=180-30] (front_se);
    \draw[dashed] (front_nw) to[out=30,in=180-30] (front_ne);
  \end{tikzpicture}
}  
\NewDocumentCommand{\DeltaTwoCubeD}{ O{} O{} O{} O{} O{} O{} O{} O{} O{1} }{  
  \begin{tikzpicture}[myStyle, scale=#9*1]
    \CubeDUnwrapped{#1}{#2}{#3}{#4}{#5}{#6}{#7}{#8}
    \draw (back_sw) to[out=180+30,in=180] node[auto,swap]{$\delta$} (front_sw)
      (back_se) to[out=0,in=30] (front_se)
      (back_nw) to[out=180+30,in=180] (front_nw);
    \draw[dashed] (back_ne) to[out=0,in=30] (front_ne);
  \end{tikzpicture}
}  

\begin{document}
\title[]{Some notes on the ternary modular commutator}
\author{Andrew Moorhead \\ }
\date{\today}

\address[Andrew Moorhead]{
  Department of Mathematics;
  Vanderbilt University;
  Nashville, TN;
  U.S.A.}
\email[Andrew Moorhead]{andrew.p.moorhead@vanderbilt.edu}

\thanks{This material is based upon work supported by the National
Science Foundation grant no.\ DMS 1500254}

\begin{abstract}
We define a relation that describes the ternary commutator for congruence modular varieties. Properties of this relation are used to investigate the theory of the higher commutator for congruence modular varieties. 
\end{abstract}

\maketitle

\section{Introduction}
The topic of these notes is commutator theory. Specifically, we study higher commutators which are a higher arity generalization of the classical binary commutator for Mal'cev varieties that was discovered by Smith in \cite{smith}. Smith's commutator was extended to modular varieties by Hagemann and Hermann in \cite{HagHer}. The theory of the modular commutator is developed in detail by Freese and McKenzie in \cite{FM} and Gumm in \cite{Gumm}. Commutator theory has been investigated for varieties that are not modular, notably by Kearnes and Szendrei in \cite{kearnesszendrei} and Kearnes and Kiss in \cite{KissKearnes}.

The definition of the binary commutator was generalized by Bulatov in \cite{bulatov proc} to a definition of a commutator of higher arity. The basic properties of this higher commutator were developed for congruence permutable varieties by Aichinger and Mudrinski in \cite{Aich Mun}. The author of \cite{moorhead} showed many of the basic properties true of the higher commutator in permutable varieties hold also for modular varieties. In \cite{wires}, Wires develops several properties of higher commutators outside of the context of modularity. 

Our main objective here is to show that the development of the binary commutator for congruence modular varieties is a special case of a more general development of a commutator of arity $\geq 2$ for congruence modular varieties, although we present here only the ternary case. Specifically, the binary modular commutator is shown in \cite{FM} to be equivalently defined in three ways: 
\begin{enumerate}
\item
as the least congruence satisfying the term condition, 
\item
as the union of classes related to the diagonal by a special congruence called $\Delta$, or 

\item the greatest binary operation on all congruence lattices across a variety satisfying certain conditions. 
\end{enumerate}
The higher commutator was defined by Bulatov as a higher arity operation obtained by generalizing the term condition. In \cite{oprsal}, Opr{\v s}al develops the theory of the higher commutator for permutable varieties by finding the right way to define $\Delta$ in the higher arity case. We extend this idea to the theory of the ternary commutator for a modular variety. 

The primary value of this treatment of the modular ternary commutator is that it involves a detailed and nontrivial example of a type of relation that we call a `higher dimensional congruence.' Higher dimensional congruences are used to develop properties of the higher commutator for Taylor varieties in \cite{moorheadinprep}.

The structure of the note is as follows: In Sections 2-4 we develop notation, list the properties of modularity that are to be used and briefly review properties of $\Delta$ for the binary case. In Section 5 we introduce the idea of a complex of matrices. In Section 6 these complexes of matrices are used to define a new type of centrality which is shown to be equivalent to the usual term condition centrality in a modular variety. In Section 7 we define $
\Delta$ and use it to characterize the ternary commutator. In Section 8 we show that the ternary commutator is related to the binary commutator by the inequality

$$ [[\theta_0, \theta_1], \theta_2] \leq [\theta_0, \theta_1, \theta_2].$$

In the final section we show that the ternary commutator and $3$-dimensional cube terms are related in a manner which is analogous to the relationship between the binary commutator and a Mal'cev operation.

\section{Matrices and Centrality}

In this section we define matrices. Relations that are comprised of such matrices can be used to define the Bulatov or what we call the \textbf{term condition commutator}. Although this note deals with the ternary commutator, the definitions given here are for arbitrary arity. 

\subsection{Dimension Two}
We begin with $(2)$-dimensional matrices. The binary commutator is defined with the so-called term condition. Let $\A$ be an algebra and let $\theta_0, \theta_1 , \delta \in \A$. We say that \textbf{$\theta_0$ centralizes $\theta_1$ modulo $\delta$} if 

\[\text{for all $t \in \Pol(\A)$, $\textbf{a}_0 \equiv_{\theta_0} \textbf{b}_0$ and $\textbf{a}_1 \equiv_{\theta_1} \textbf{b}_1$},
\]

\[
t(\textbf{a}_0, \textbf{a}_1) \equiv_\delta t(\textbf{a}_0, \textbf{b}_1 ) \implies t(\textbf{a}_0, \textbf{b}_1) \equiv_\delta t(\textbf{b}_0, \textbf{b}_1 ).
\]

This is the standard definition of the term condition, but the term condition can be formulated in terms of matrices, which we shall now define. The set $2^2$ is a natural coordinate system for a square and we say that a relation $R \subseteq A^{2^2}$ is \textbf{$(2)$-dimensional}. We will usually orient our squares like

\[
\SQuare[(0,0)][(0,1)][(1,0)][(1,1)][1.3],
\]
so that when we write 

\[
h =\SQuare[a][c][b][d] \in R
\]
for some $(2)$-dimensional relation $R$, we mean that $h_{(0,0)} = a$, $h_{(1,0)} = b$ and so on. 

Now define the \textbf{algebra of $(\theta_0, \theta_1)$-matrices} to be 

\[
M(\theta_0, \theta_1) \coloneqq \Sg_{\A^{2^2}}\left( \left\{\SQuare[x][x][y][y]: \langle x,y \rangle \in \theta_0 \right\} \union \left\{ \SQuare[x][y][x][y]: \langle x,y \rangle \in \theta_1 \right\} \right).
\]
The centrality condition $C(\theta_0, \theta_1; \delta)$ is equivalent to the assertion
\[
\forall \SQuare[a][c][b][d] \in M(\theta_0, \theta_1)
\left( 
\DeltaZeroSQuare[a][c][b][d]
\right),
\]
where the diagram in parenthesis is a picture representing the implication 

\[\langle a,c \rangle \in \delta \implies \langle b,d \rangle \in \delta.
\]
Similarly, $C(\theta_1, \theta_0; \delta)$ is equivalent to the assertion
\[
\forall \SQuare[a][c][b][d] \in M(\theta_0, \theta_1)
\left( 
\DeltaOneSQuare[a][c][b][d]
\right).
\]
Now, set

\[
[\theta_0, \theta_1] \coloneqq \Meet \{\delta: C(\theta_0, \theta_1; \delta) \}.
\]

\subsection{Arbitrary Dimension}

For higher arity commutators we prefer at the outset to define term condition centrality and the commutator in terms of matrices. The definition we give is easily seen to be equivalent to the usual one, see \cite{moorhead}.

Let $\A$ be an algebra and let $k \in \nat_{\geq 1}$. The set of functions $2^k$ is a coordinate system for the $(k)$-dimensional hypercube, where two functions are connected by an edge if and only if they differ in exactly one coordinate. An $\A$-invariant relation 

\[
R \leq  A^{2^k}
\]
is called a \textbf{$(k)$-dimensional invariant relation}. Members of such an $R$ are what we call \textbf{vertex labeled hypercubes}, although we usually not be so formal.

Take an $m \in A^{2^k}$. Notation is needed to single out factors of $m$ that are lower dimensional matrices. We specify a partial function $g: k \rightarrow 2$ by a pair of tuples, $D = (d_0, \dots, d_{l-1})$ and $R = (g(d_0), \dots, g(d_{l-1}))$, where the tuple $D$ lists the elements of the domain of $g$ in order and $|\text{dom}(g)| = l$. Let the map 

\[
\crsec_D^R :  A^{2^k} \rightarrow A^{2^{k\setminus{D}}}
\]
be the projection onto those factors indexed by functions $f \in 2^k$ that extend $g$. Formally, $(\crsec_D^R(m))_h = m_f$, where $f = g \union h$. In case $|D| = k-2$ or $|D| = k-1$ we denote the map $\crsec^R_D$ by  $\Square_D^R$ or $\Line_D^R$, respectively. When $|D| = 1$ the map will be called $\face_D^R$. 

Suppose $|D| = k-1$ and let $i \in k \setminus D$. When $R$ is not the constant $1$ tuple we call $\Line_D^R(m)$ an \textbf{$(i)$-supporting line}, and when $R$ is the constant $1$ tuple we call $\Line_D^R(m)$ the \textbf{$(i)$-pivot line}. Similarly, suppose $|D| = k-2$ and take $ i\neq j \in k \setminus D$. When $R$ is not the constant $1$ tuple we call $\Square_D^R(m)$ an \textbf{ $(i,j)$-supporting square}, and when $R$ is the constant $1$ tuple we call $\Square_D^R(m)$ the \textbf{$(i,j)$-supporting square}.

 A picture of a matrix, pair of squares and a line is given in Figure \ref{fig:cubeandfaces}. Notice the orientation of the cube with respect to the coordinates, as this orientation is used throughout this note.

\begin{figure}

\includegraphics{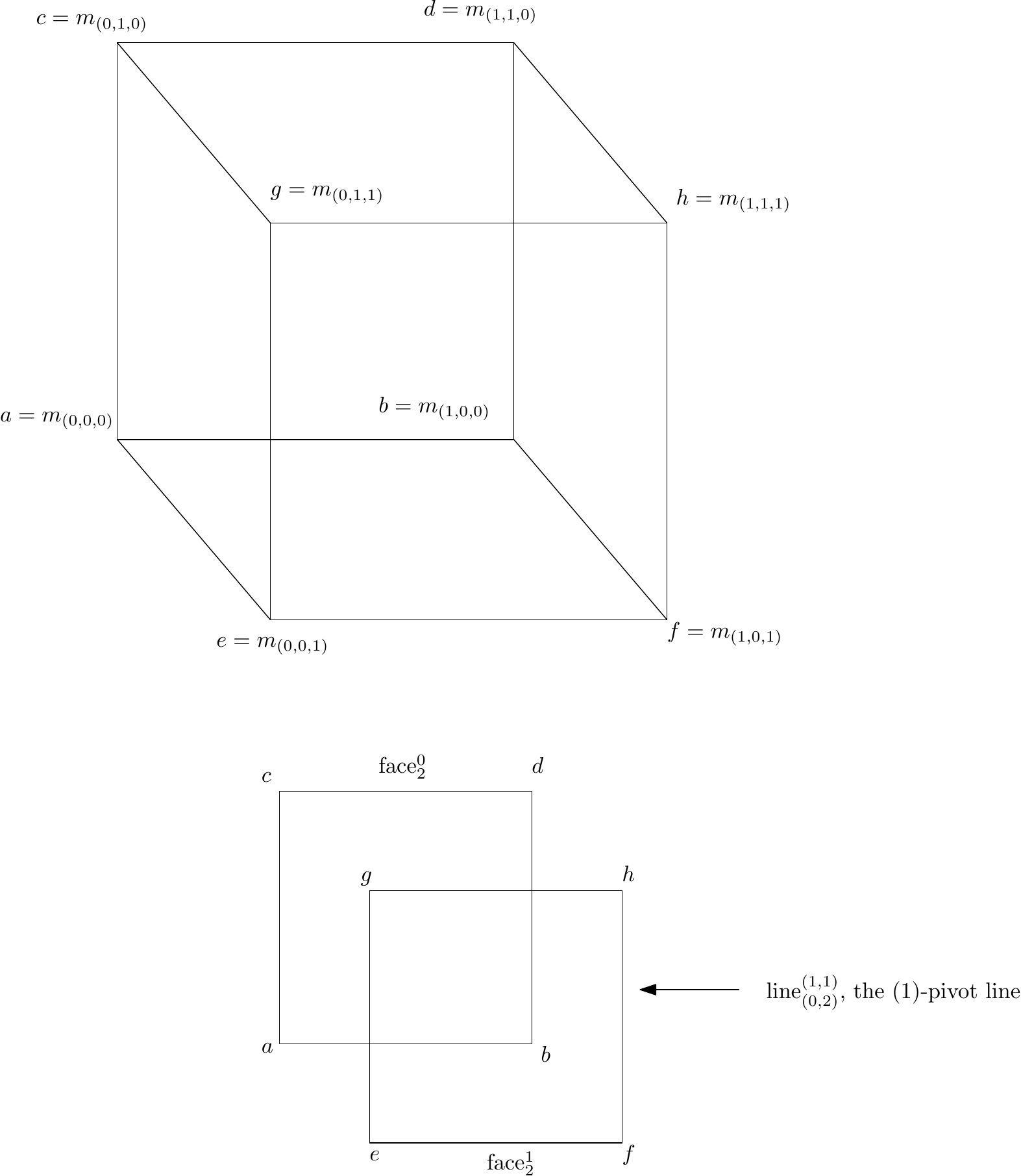}

\caption{}\label{fig:cubeandfaces}
\end{figure}

The term condition definition of higher centrality can now be formulated in terms of matrices. Let $\cube_i^k(a,b) \in A^{2^k}$ be the matrix such that $\face_i^0(\cube_i^k(a,b))$ and $\face_i^1(\cube_i^k(a,b))$ are constant matrices with value $a$ and $b$, respectively. Now let $T = (\theta_0, \dots, \theta_{k-1}) \in \Con(\A)^k$ be a tuple of congruences and set 

\[
M(\theta_0, \dots, \theta_{k-1}) = \Sg_{\A^{2^k}} (\{ \cube^k_i(a,b) : i \in k \text{ and } \langle a, b \rangle \in \theta_i\})
\]
We call $M(\theta_0, \dots, \theta_{k-1})$ the \textbf{algebra of $(\theta_0, \dots, \theta_{k-1})$-matrices}. If no confusion can result, we will refer to $M(\theta_0, \dots, \theta_{k-1})$ as $M(T)$, and call its elements $T$-matrices.

Now we can define higher centrality, see \cite{moorhead} for the same definition. 
\begin{defn}\label{def:cen}
Let $\A$ be an algebra, $k \geq 2$, $T = (\theta_0, \dots, \theta_{k-1}) \in \Con(\A)^k$, and $\delta \in \Con(\A)$. For $j\in k$, we say that \textbf{$T=(\theta_0, \dots, \theta_{k-1})$ is centralized at $j$ modulo $\delta$} if the following property holds for all $(\theta_0, \dots, \theta_{k-1})$-matrices $m$:
\begin{enumerate}
\item[(*)] If every $(j)$-supporting line of $m$ is a $\delta$-pair, then the $(j)$-pivot line of $m$ is a $\delta$-pair. 

\end{enumerate}
We write $C(T;j;\delta)$ to indicate that $T$ is centralized at $j$ modulo $\delta$.
\end{defn}

Now, set

\[
[T]_j \coloneqq \Meet \{\delta: C(T;j; \delta) \}.
\]

\section{Day Terms}\label{DayTerms}
The following classical results on congruence modular varieties are needed. For proofs see \cite{Day}, \cite{Gumm} or \cite{FM}. 
\begin{prop}[Day Terms]\label{prop:day}
A variety $\var$ is congruence modular if and only if there exist term operations $m_e(x,y,z,u)$ for $e\in n+1$ satisfying the following identities:

\begin{enumerate}
\item $m_e(x,y,y,x) \approx x$ for each $0\leq e \leq n$,
\item $m_0(x,y,z,u) \approx x$,
\item $m_n(x,y,z,u) \approx u$,
\item $m_e(x,x,u,u) \approx m_{e+1}(x,x,u,u) $ for even $e$, and
\item $m_e(x,y,y,u) \approx m_{e+1}(x,y,y,u)$ for odd $e$

\end{enumerate}
\end{prop}

\begin{prop}[Lemma 2.3 of \cite{FM}]\label{prop:shift}
Let $\var$ be a variety with Day terms $m_e$ for $e\in n+1$. Take $\delta \in \Con(\A)$ and assume $\langle b,d \rangle  \in \delta$. For a tuple $\langle a,c \rangle \in A^2$ the following are equivalent:

\begin{enumerate}
\item $\langle a,c \rangle \in \delta$
\item $\langle m_e(a,a,c,c) , m_e(a,b,d,c) \rangle \in \delta$ for all $e\in n+1$
\end{enumerate}
\end{prop}

\section{Binary Commutator and the $\Delta$ relation}

The binary commutator of two congruences $\theta_0, \theta_1$ of $\A$ can be developed by examining a special congruence called $\Delta_{\theta_0, \theta_1}$, which is a congruence of $\theta_0$, considered here to be a subalgebra of $\A^2$. The definition of $\Delta_{\theta_0, \theta_1}$ is as follows: 

$$ \Delta_{\theta_0, \theta_1} = \Cg^{\theta_0}(\{ \langle \langle x,x\rangle, \langle 
y,y \rangle \rangle : \langle x,y \rangle \in \theta_1\})$$

It is informative to use the matrices and coordinates introduced earlier. With this perspective in mind, $\theta_0$-pairs become $1$-dimensional matrices, which according to our convention are horizontal lines. The pair $ \langle a, b \rangle$ is $\Delta_{\theta_0, \theta_1}$ related to $\langle c, d \rangle $ if there is a sequence of matrices as in Figure \ref{fig:delta_1}. That is, $\Delta_{\theta_0, \theta_1}$ is the transitive closure of $M(\theta_0, \theta_1)$, where these matrices are considered to be pairs of $\theta_0$-pairs.

Of course, one can also define 

$$\Delta_{\theta_1, \theta_0} = \Cg^{\theta_1}(\{ \langle \langle x,x\rangle, \langle 
y,y \rangle \rangle : \langle x,y \rangle \in \theta_0\})$$

This congruence is the transitive closure of $M(\theta_0, \theta_1)$, where these matrices are now considered to be $\theta_1$-pairs. 

Both $\Delta_{\theta_0, \theta_1}$ and $\Delta_{\theta_1, \theta_0}$ are subalgebras of $\A^{2^2}$, see Figure \ref{fig:delta_1}.
\begin{figure}[ht]

\includegraphics{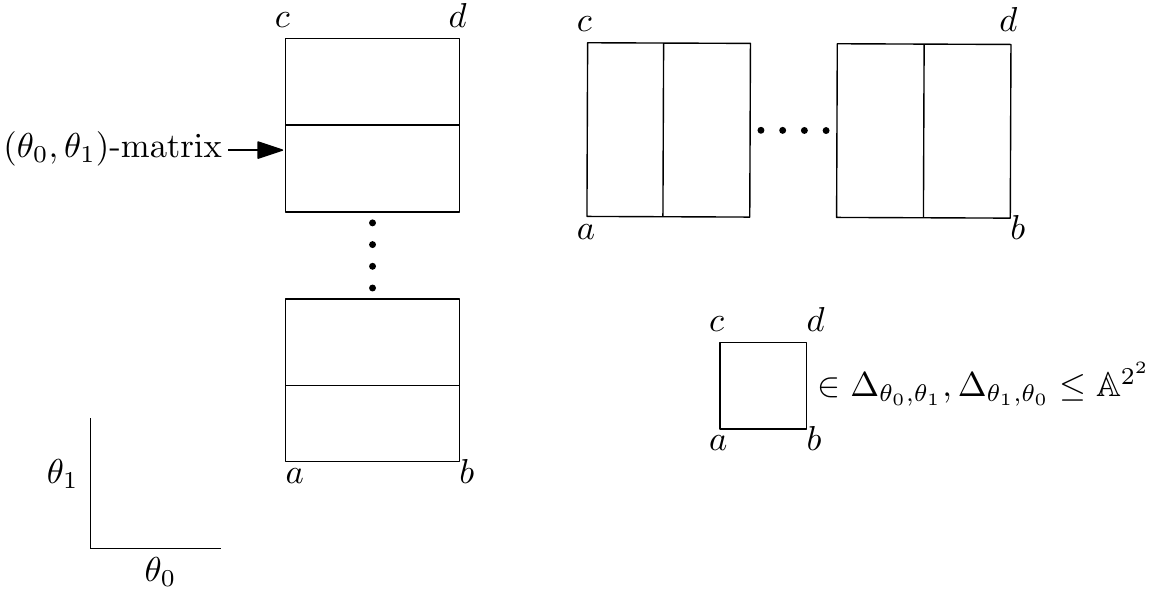}

\caption{}\label{fig:delta_1}
\end{figure}
In fact, an aspect of the symmetry of the binary commutator is given in the following theorem. 

\begin{thm}[See Chapter 4 of \cite{FM}]\label{thm:BinDelSym}

$\Delta_{\theta_0, \theta_1} = \Delta_{\theta_1, \theta_0}$

\end{thm}

The $\Delta$ relation provides a nice characterization of the binary commutator of $\theta_0, \theta_1$.

\begin{thm}[Theorem 4.9 of \cite{FM}]\label{thm:DelBin}
For $x, y \in \A$ the following are equivalent:

\begin{enumerate}
\item $\langle x,y \rangle \in [x,y]$

\item $\left[ \begin{array}{cc}
				x&y\\
				y&y
				\end{array}
				\right] \in \Delta_{\theta_0, \theta_1}$ \smallskip
\item $\left[ \begin{array}{cc}
				x&b\\
				y&b
				\end{array}
				\right] \in \Delta_{\theta_0, \theta_1}$
for some $b\in \A$\smallskip
\item $\left[ \begin{array}{cc}
				x&y\\
				c&c
				\end{array}
				\right] \in \Delta_{\theta_0, \theta_1}$
for some $c\in \A$.

\end{enumerate}

\end{thm}

The first portion of this note is devoted to extending Theorem \ref{thm:DelBin} to the ternary commutator.

\section{Complexes of Matrices}

Let $\A$ be an algebra and take $(\theta_0, \theta_1, \theta_2) \in \Con(\A)^3$. We begin this section defining what we call a $(\theta_0, \theta_1, \theta_2)$-\textbf{matrix complex}. Informally, these are blocks that are built from $(\theta_0, \theta_1, \theta_2)$-matrices. Formally, take some integers $n_0, n_1, n_2 \geq 2$, and set 
$$ \mathcal{C}_{n_0, n_1, n_2} =  A^{n_0\times n_1\times n_2}$$
Now take 
$$\mathcal{C} = \Union_{n_0, n_1, n_2 \geq 2} \mathcal{C}_{n_0, n_1, n_2}$$
We call $\mathcal{C}$ the set of \textbf{$3$-dimensional complexes} of the algebra $A$, and $C_{n_0, n_1, n_2}$ the set of \textbf{complexes with dimensions $(n_0, n_1, n_2)$}.

For $ f=(f_0, f_1, f_2) \in n_0\times n_1 \times n_2$ and $g = (g_0, g_1, g_2) \in 2^3 = 2\times 2 \times 2$ let $f+g = (f_0 +g_0, f_1+g_1, f_2 +g_2)$. Let $f = (f_0, f_1, f_2) \in n_0 \times n_1 \times n_2 $ be such that for each $i \in 3$ we have $f_i < n_i-1$. Now define 
$$ \Mat_f : \mathcal{C}_{n_0, n_1, n_2} \rightarrow A^{2^3}$$ 
to be the map defined by 

$$  (\Mat_f(a))_g = a_{f+g}.$$

The image of such a map is called a \textbf{component matrix}. Now define a \textbf{$(\theta_0, \theta_1, \theta_2)$-matrix complex with dimensions $(n_0, n_1, n_2)$} to be a complex $a \in \mathcal{C}_{n_0, n_1, n_2}$ with the property that $  \Mat_f(a) \in M(\theta_0, \theta_1, \theta_2)$ for each map $ \Mat_f $. The collection of all $(\theta_0, \theta_1, \theta_2)$-complexes with dimensions $(n_0, n_1, n_2)$ is called $\mathcal{C}_{n_0, n_1, n_2}(\theta_0, \theta_1, \theta_2)$ and the collection of all $(\theta_0, \theta_1, \theta_2)$-complexes is called $\mathcal{C}(\theta_0, \theta_1, \theta_2)$. 

We can take cross-sections of complexes in the same way as we do for matrices. Consider $c \in \mathcal{C}_{n_0, n_1, n_2}$. Let $D$ be a tuple that lists some coordinates in order and $R$ be a tuple that specifies a value that is assigned to each corresponding coordinate in $D$. Set $\crsec_D^R(c)$ to be the lower dimensional complex obtained by projecting $c$ onto those factors indexed by functions that extend the partial function specified by $D$ and $R$. We use the names $\Line$ and $\Square$ in the same way as for complexes. A typical complex, one of its component matrices and a cross-section square are shown in Figure \ref{fig:complex_1}.
\begin{figure}[ht]

\includegraphics{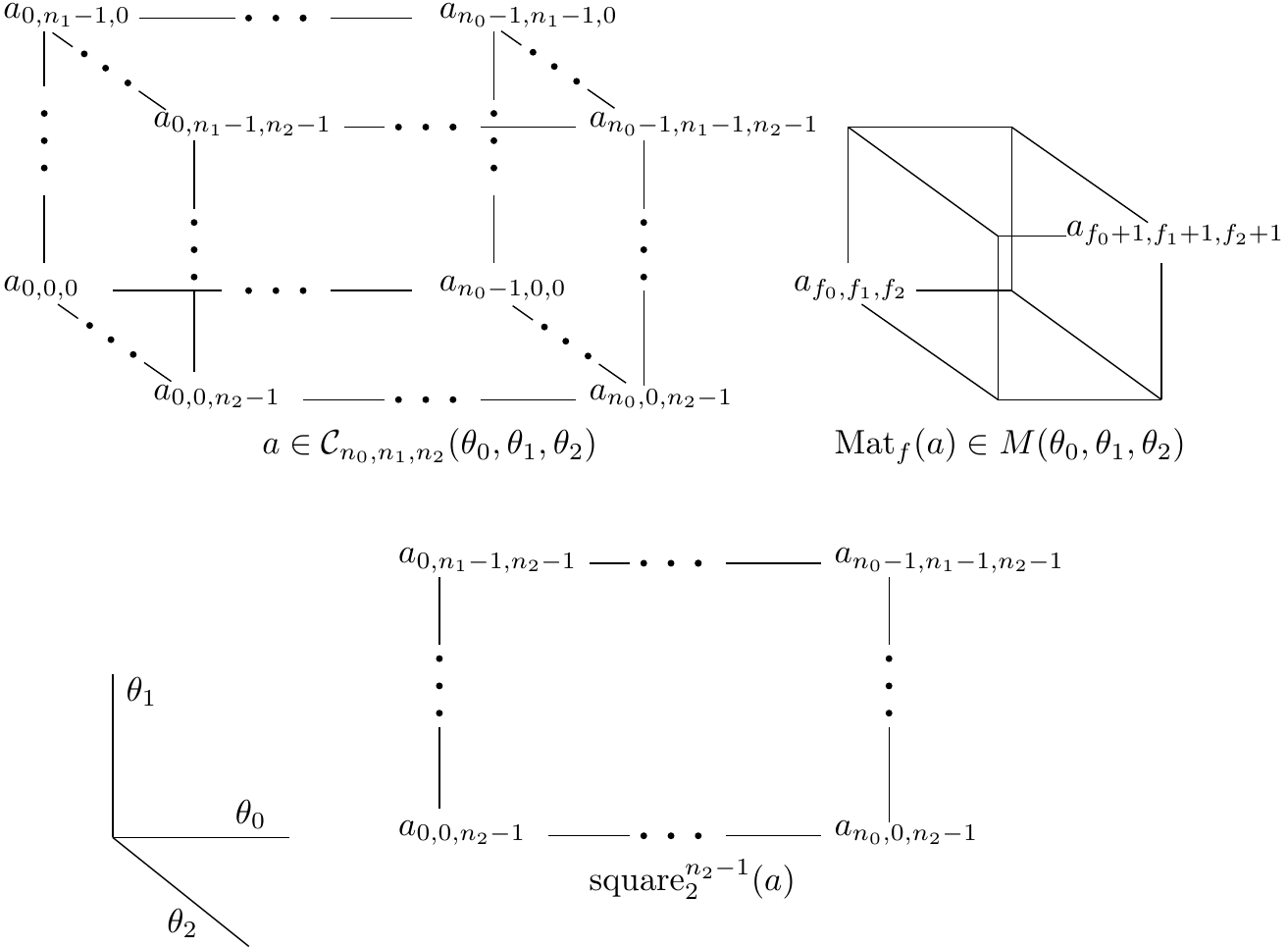}

\caption{}\label{fig:complex_1}
\end{figure}

We now define a map that allows us to identify the corners of a complex. For integers $n_0, n_1, n_2 \geq 2$ and a function $g \in 2^3$, let $\gamma(g) \in n_0 \times n_1 \times n_2$ be the tuple that is $0$ or $n_i-1$ in the $i$-th coordinate if $g(i)$ is $0$ or $1$, respectively. Then set 

$$\Corner_{n_0, n_1, n_2} : \mathcal{C} \rightarrow A^{2^3} $$ to be the map defined by

$$\Corner_{n_0, n_1, n_2}(a)_g = a_{\gamma(g)} $$

We will often refer to this map as $\Corner$ when no confusion over the dimensions of a complex is possible. Finally, we define $$\Corners_{n_0, n_1, n_2}(\theta_0, \theta_1, \theta_2) = \{\Corner(a) : a \in \mathcal{C}_{n_1, n_2, n_3}(\theta_0, \theta_1, \theta_2) \}$$
and we denote the set of all $(\theta_0, \theta_1, \theta_2)$-corners by $\Corners(\theta_0, \theta_1, \theta_2)$.

Notice that because $M(\theta_0, \theta_1, \theta_2) = \mathcal{C}_{2,2,2}(\theta_0, \theta_1, \theta_2) = \Corners_{2,2,2}(\theta_0, \theta_1, \theta_2)$, we have that $M(\theta_0, \theta_1, \theta_2) \subseteq \Corners(\theta_0, \theta_1, \theta_2)$. 

\section{Transitive Term Condition}
We now define a stronger term condition which we call the \textbf{transitive term condition}.

\begin{defn}\label{def:trancen}
We say that \textbf{$(\theta_0, \theta_1, \theta_2)$ is transitively centralized at $j$ modulo $\delta$}  if the following property holds for all $(\theta_0, \theta_1, \theta_2)$-corners $c$:
\begin{enumerate}
\item[(*)] If every $(j)$-supporting line of $c$ is a $\delta$-pair, then the $(j)$-pivot line of $c$ is a $\delta$-pair. 

\end{enumerate}

We abbreviate this property $C_{tr}(\theta_0, \theta_1, \theta_2; j; \delta)$.
\end{defn}

Because every element of $M(\theta_0, \theta_1, \theta_2)$ is a corner, we have that $C_{tr}(\theta_0, \theta_1, \theta_2; j ; \delta)$ implies $C(\theta_0, \theta_1, \theta_2; j; \delta)$. One of the results of \cite{moorhead} is that the ordinary term condition is symmetric, so we may unambiguously write $C(\theta_0, \theta_1, \theta_2 ; \delta)$. As it turns out, the transitive term condition is equivalent to the usual term condition in a modular variety.

\begin{lem}\label{lem:comrot}
Let $\var$ be a modular variety with Day terms $m_0, \dots, m_n$. Take $\A \in \var$ and $(\theta_0, \theta_1, \theta_2) \in \Con(\A)^3$. Choose some coordinates $i \neq j \neq l \in 3$. Take $z \in \mathcal{C}_{n_0, n_1, n_2}(\theta_0, \theta_1, \theta_2) $ such that 

$$\Corners(\Square_i^0) =\left[ \begin{array}{cc}
				c_0& d_0	\\			
				a_0& b_0
				\end{array}
				\right] \text { and }$$ 
$$\Corners(\Square_i^{m_i-1}) =\left[ \begin{array}{cc}
				c_{m_i-1}& d_{m_i-1}	\\			
				a_{m_i-1}& b_{m_i-1}
				\end{array}
				\right], $$ 
where the matrices above are oriented so that rows are indexed by $j$ and the columns are indexed by $l$, with the origin at the bottom left.

For each Day term $m_e $ with $e\in n+1$ there is an $R^e_{j,l}(z) \in \mathcal{C}(\theta_0, \theta_1, \theta_2)$ with dimensions $m_i \geq n_i$, $m_j= \max(n_j -1, 2)$ and $m_l = n_l$, such that 

$$\Corners(\Square_i^0) =\left[ \begin{array}{cc}
				m_e(c_0, c_0, a_0, a_0) & m_e(c_0,d_0,b_0, a_0)	\\			
				c_0& c_0
				\end{array}
				\right] \text{ and }$$
$$\Corners(\Square_i^{m_i-1}) =\left[ \begin{array}{cc}
				m_e(c_{n_i-1}, c_{n_i-1}, a_{n_i-1}, a_{n_i-1}) & m_e(c_{n_i-1},d_{n_i-1},b_{n_i-1}, a_{n_i-1})	\\			
				c_{n_i-1}& c_{n_i-1}
				\end{array}
				\right] $$

We call $R^e_{j,l}(z)$ \textbf{the $e$th shift rotation at $(j,l)$} of $z$.
\end{lem}
\begin{proof}
Without loss, let $i=2$, $l=0$ and $j=1$. Take $z \in \mathcal{C}_{n_0, n_1, n_2}(\theta_0, \theta_1, \theta_2)$. Now, $z$ may be decomposed into the sequence of cross-sections $\Square_2^k(z)$ for $k \in n_2$ as shown in Figure \ref{fig:rotcom1}. 

\begin{figure}[ht]

\includegraphics[scale=.9]{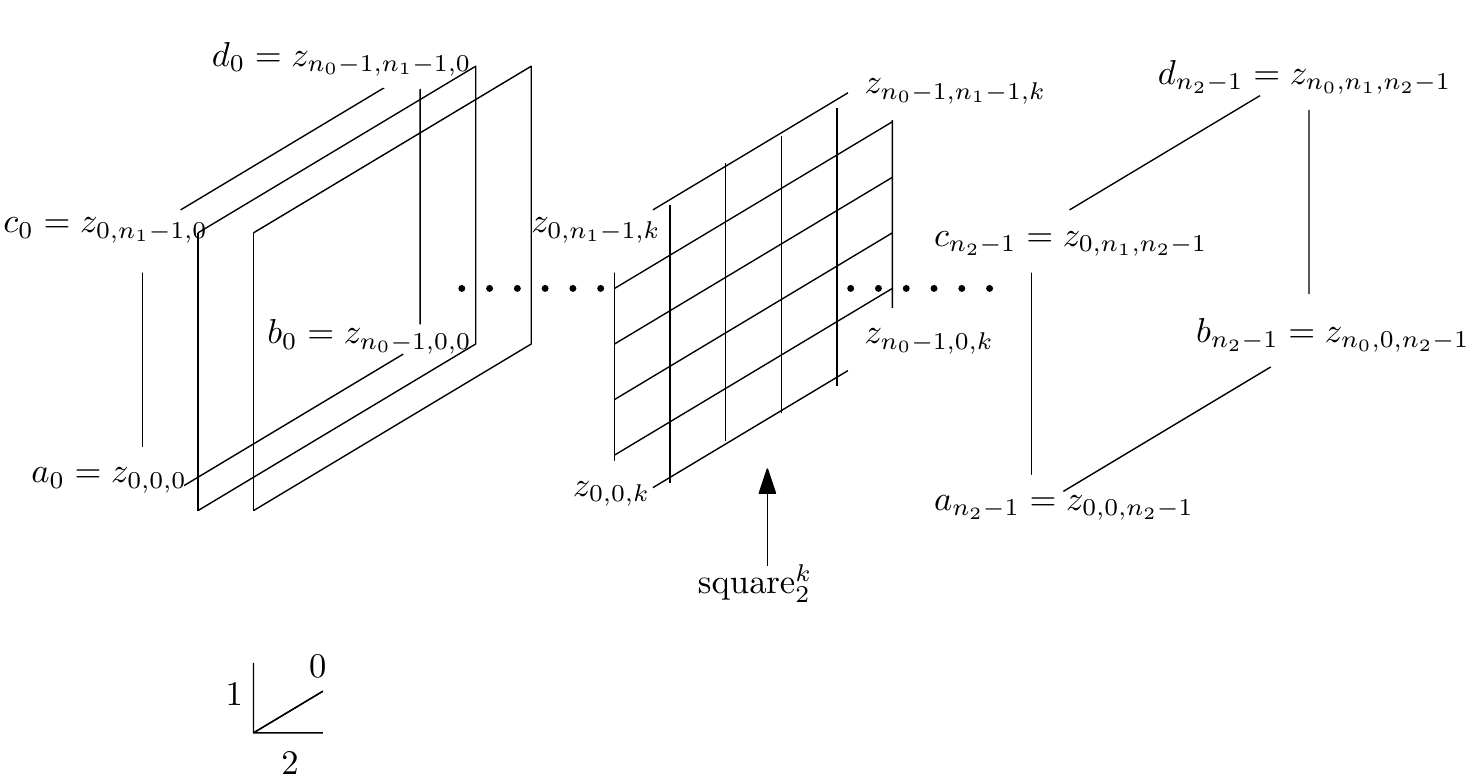}

\caption{}\label{fig:rotcom1}
\end{figure}

For each cross-section $\Square_2^k(z)$, let $A_k, B_k, C_k$ and $D_k$ be the following $(\theta_0, \theta_1)$-complexes (see Figure \ref{fig:rotcom2}): 
\begin{enumerate}
\item $A_k$ takes the value $z_{0, n_1-1, k}$ in each factor.
\item $B_k$ is equal to $\Square^k_2(z)$.
\item $C_k$ is such that 
\begin{itemize}
\item $\Line_1^0(C_k) = \Line_1^{n_1-1}(C_k) = \Line_1^0(\Square^k_2(z))$ and 

\item $\Line_1^k(C_k) = \Line_1^1(\Square^k_2(z))$ for all other $k \in n_1$.
\end{itemize}
\item $D_k$ is such that 
\begin{itemize}
\item
$\Line_1^k(D_k)$ is constantly $z_{0, n_1-1, k}$ for $k \in n_1-1$ and 
\item
$\Line_1^{n_1-1}(D_k)$ is constantly $z_{0,0,k}$.
\end{itemize}
\end{enumerate}

\begin{figure}[ht]

\includegraphics[scale=.85]{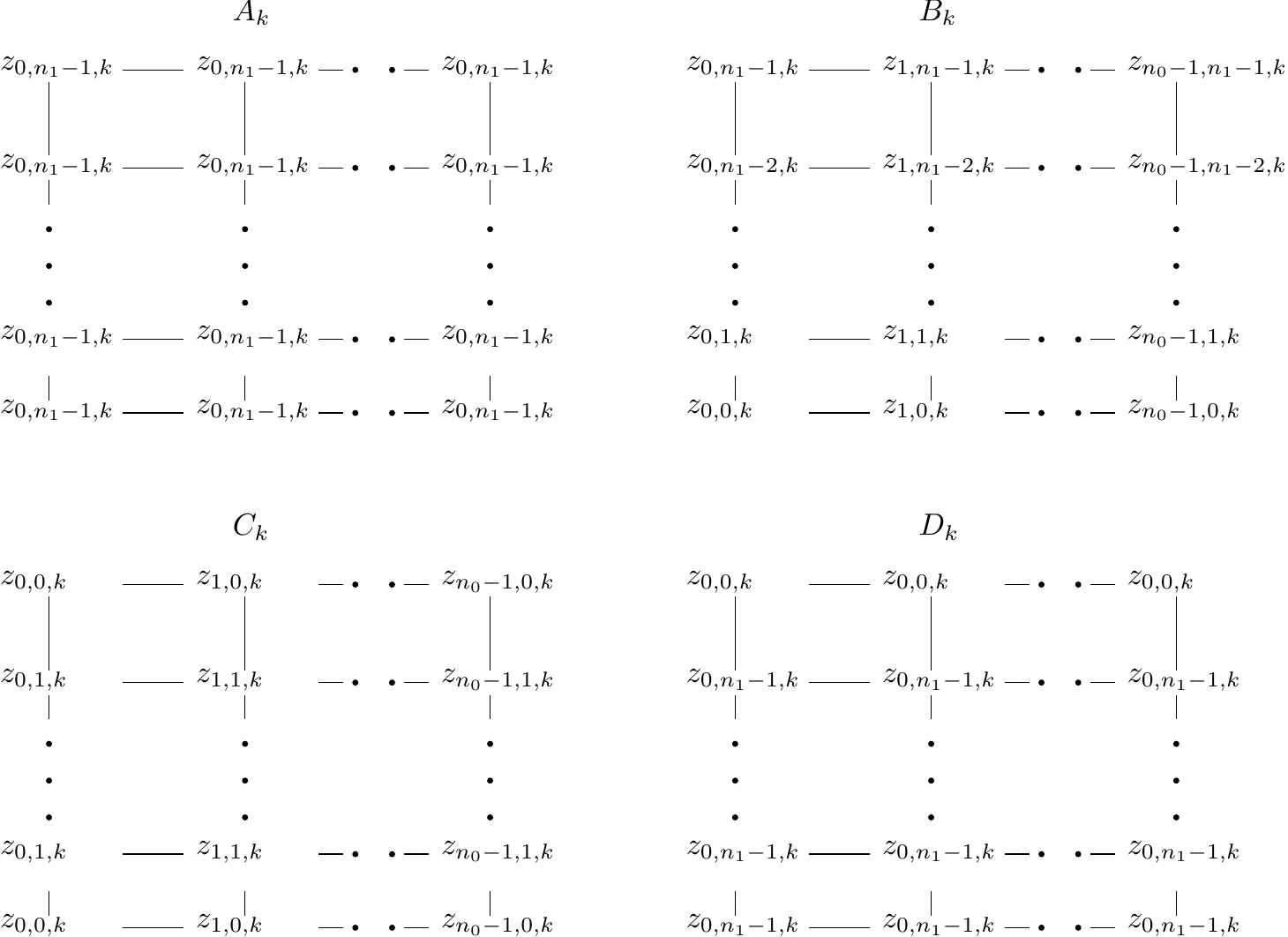}

\caption{}\label{fig:rotcom2}
\end{figure}

Let $E_k =m_e(A_k, B_k,C_k, D_k)$ and let $F_k \in \mathcal{C}_{n_0,n_1-1}(\theta_0, \theta_1)$ be the complex such that $(F_k)_g = (E_k)_{g+(0,1)}$ for all $g \in n_0\times (n_1-1)$. So, $F_k$ is obtained by deleting the bottom row of $E_k$. Applying the identity $m_e(x,y,y,x) \approx x$, we get that $\Corners(F_k)$ is the $(\theta_0, \theta_1)$-matrix shown in \ref{fig:rotcom3}.
\begin{figure}[ht]

\includegraphics{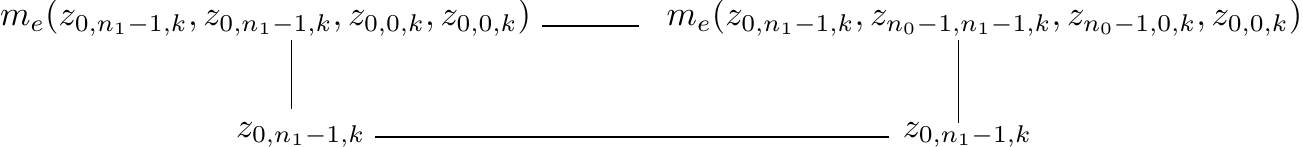}

\caption{}\label{fig:rotcom3}
\end{figure}

We will be done if we can show that there is a complex $y \in \mathcal{C}_{n_0, n_1, m_2}$, for some $m_2 \geq n_2$, such that $\Square_2^0 = E_0$ and $\Square_2^{m_2-1} = E_{n_2-1}$. We will demonstrate this by showing that there are $\alpha, \beta, \gamma, \epsilon \in \mathcal{C}_{n_0, n_1, m_2}$ such that 
\begin{enumerate}
\item $\Square_2^0{\alpha} = A_0$ and $\Square_2^{m_2-1}{\alpha} = A_{m_2-1}$
\item $\Square_2^0{\beta} = B_0$ and $\Square_2^{m_2-1}{\beta} = B_{m_2-1}$
\item $\Square_2^0{\gamma} = C_0$ and $\Square_2^{m_2-1}{\gamma} = C_{m_2-1}$
\item $\Square_2^0{\epsilon} = D_0$ and $\Square_2^{m_2-1}{\epsilon} = D_{m_2-1}$
\end{enumerate}
We then set $y = m_e(\alpha, \beta, \gamma, \delta)$.

So, let $\alpha', \beta', \gamma' \in \mathcal{C}_{n_0, n_1, n_2}$ be such that for each $k \in n_2$

\begin{enumerate}
\item $\Square_2^k{\alpha'} = A_k$
\item $\Square_2^k{\beta'} = B_k$
\item $\Square_2^k{\gamma'} = C_k$.

\end{enumerate}
It is easy to see that $\alpha, \beta, \gamma \in \mathcal{C}_{n_0, n_1, n_2}$. This does not work for the $D_k$, because $$\left[ \begin{array}{cc}
				z_{0,0,k}& z_{0,0,k+1}\\
				z_{0,n_1-1,k}&z_{0,n_1-1,k+1}
				\end{array}
				\right] $$ need not be a $(\theta_2, \theta_1)$-matrix. However, it is an element of $\Delta_{\theta_2, \theta_1}$, and therefore also an element of $\Delta_{\theta_1,\theta_2}$ by Theorem \ref{thm:BinDelSym}. So, there exists for a $t_k \in \nat_{\geq 2}$ some $\theta_1$-pairs $\langle u_0^k, v_0^k \rangle, \dots, \langle u_{t_k-1}^k, v_{t_k-1}^k \rangle$ such that $\langle z_{0,0,k}, z_{0, n_1-1,k} \rangle = \langle u_0^k, v_0^k \rangle$ and $\langle z_{0,0,k+1}, z_{0, n_1-1,k+1} \rangle = \langle u_{t_k-1}^k, v_{t_k-1}^k \rangle$, with consecutive pairs forming matrices belonging to $M(\theta_2, \theta_1)$. See Figure \ref{fig:rotcom4}.
\begin{figure}[ht]

\includegraphics{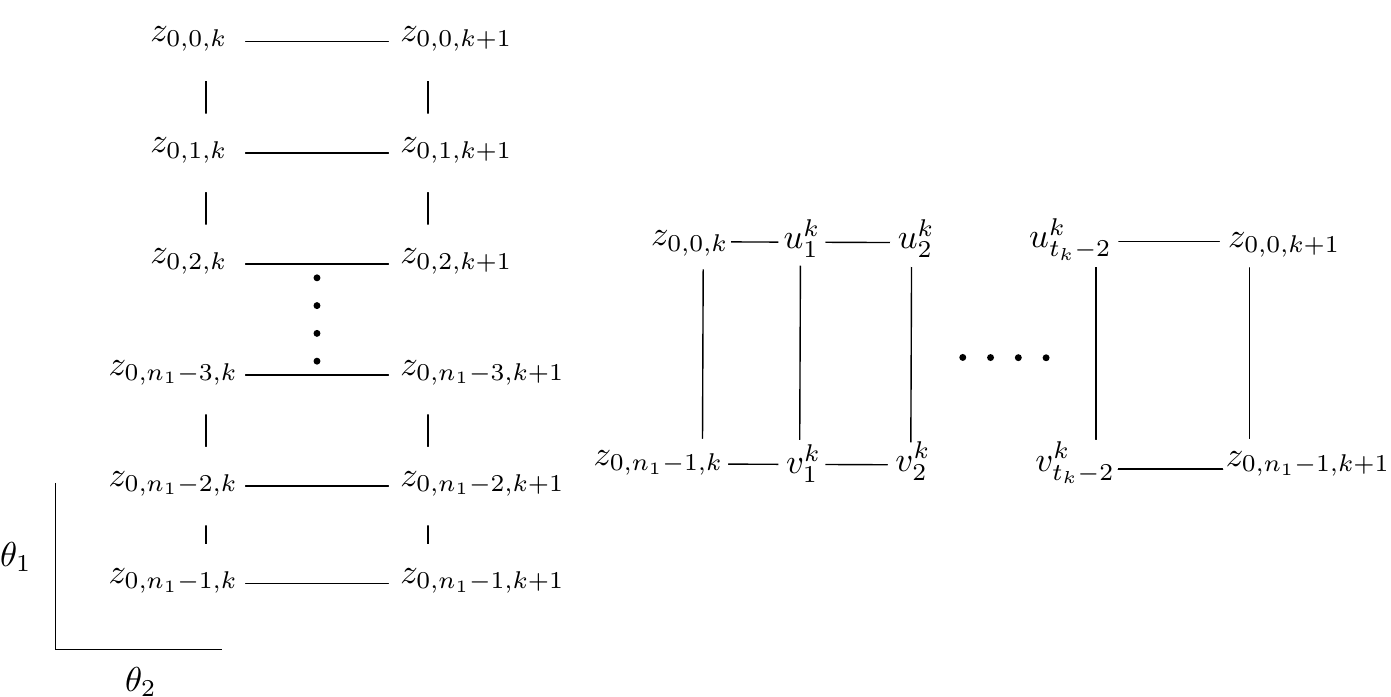}

\caption{}\label{fig:rotcom4}
\end{figure}

For each $s \in t_k$ set $U_s^k \in \mathcal{C}_{n_0, n_1}(\theta_0, \theta_1)$ to be a complex with top row equal to $u_s^k$ and all other rows equal to $v_1^k$. Let $\mathcal{U}$ be the collection of all such $U_s^k$. We order elements of $U$ by the lexicographical order on pairs $(k,s)$. Now, set $m_2 = \sum_{k \in n_2-1} t_k $ and let $\epsilon \in \mathcal{C}_{n_0, n_1, m_2}$ be such that $ \Square_2^r (\gamma) $ is the $r$th element of $\mathcal{U}$, see Figure \ref{fig:rotcom5}. It is clear that $\epsilon$ is actually an element of $\mathcal{C}_{n_0, n_1-1, m_2}(\theta_0, \theta_1, \theta_2)$.

\begin{figure}[ht]

\includegraphics{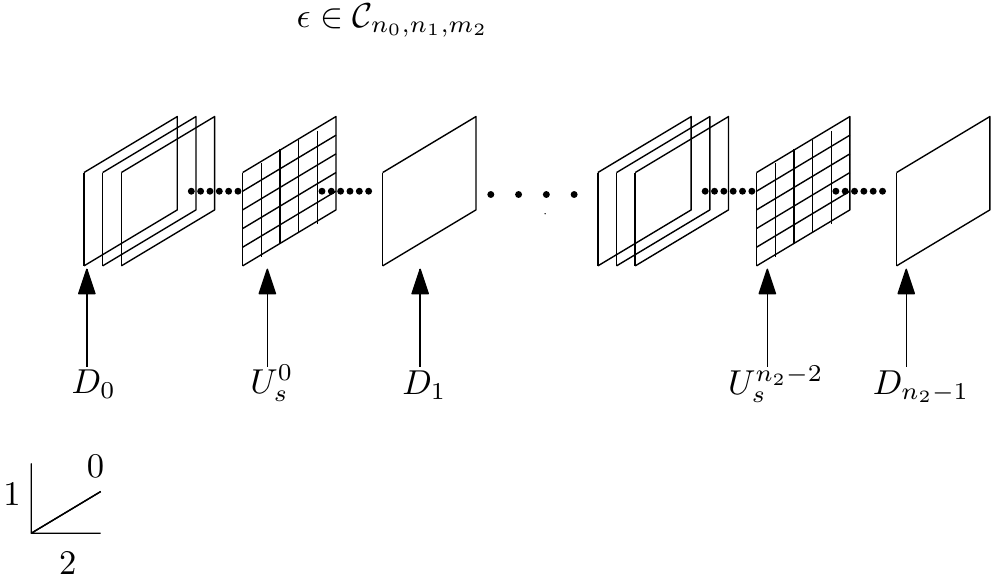}

\caption{}\label{fig:rotcom5}
\end{figure}

Now, $\alpha', \beta, \gamma'$ are $(\theta_0, \theta_1, \theta_2)$-complexes with dimensions $(n_0, n_1, n_2)$. We extend these complexes by $m_2 -n_2$ many copies of $A_{n_2-1}, B_{n_2-1}$ and $C_{n_2-1}$, respectively. That is, define $\alpha, \beta,\gamma \in \mathcal{C}_{n_0, n_1, n_2}(\theta_0, \theta_1, \theta_2)$ so that

\begin{equation*}
    \Square_2^k(\alpha) = \begin{cases}
               \Square_2^k(\alpha')              & \text{if } k \in n_2\\
               
               A_{n_2-1} & \text{otherwise}
           \end{cases}
\end{equation*}

\begin{equation*}
    \Square_2^k(\beta) = \begin{cases}
               \Square_2^k(\beta')              & \text{if } k \in n_2\\
               
               B_{n_2-1} & \text{otherwise}
           \end{cases}
\end{equation*}

\begin{equation*}
    \Square_2^k(\gamma) = \begin{cases}
               \Square_2^k(\gamma')              & \text{if } k \in n_2\\
               
               C_{n_2-1} & \text{otherwise}
           \end{cases}
\end{equation*}

Now, set $y = m_e(\alpha, \beta, \gamma, \epsilon) \in \mathcal{C}_{n_0, n_1, m_2}(\theta_0, \theta_1, \theta_2)$. Notice that $\Square_2^0(y) = E_0$ and $\Square_2^{m_2-1}(y) = F_0$. Set $R^e_{1,0}(z) \in \mathcal{C}_{n_0, n_1-1, m_2}(\theta_0, \theta_1, \theta_2)$ to be the complex defined by 

\begin{equation*}
R^e_{1,0}(z)_f = y_{f+(0,1,0)} \hspace{5mm}  \text{for } f \in n_0\times (n_1 -1) \times m_2
\end{equation*}
So, $x$ is the complex that is obtained by deleting the bottom of face of $y$. We have that $\Square_2^0(R^e_{1,0}(z)) = F_0$ and $\Square_2^{m_2-1}(R^e_{1,0}(z)) = F_{n_2-1}$, so we are done.
\end{proof}

We can say more in case the dimensions of a complex are minimal in two coordinates.

\begin{lem}\label{lem:rotwhen2}
Let $\var$ be a modular variety with Day terms $m_0, \dots, m_n$. Take $\A \in \var$ and $(\theta_0, \theta_1, \theta_2) \in \Con(\A)^3$. Fix an $e \in n+1$ and choose some coordinates $i \neq j \neq l \in 3$. Take $z \in \mathcal{C}_{n_0, n_1, n_2}(\theta_0, \theta_1, \theta_2) $ such that $m_j = m_l =2$. Suppose that for each $k \in n_i$ we have

$$\Square_i^k(z) = \left[ \begin{array}{cc}
				c_k& d_k	\\			
				a_k& b_k
				\end{array}
				\right] $$ 
Then $R^e_{j,l}(z) \in \mathcal{C}_{n_0, n_1, n_2}(\theta_0, \theta_1, \theta_2)$, i.e.\ the dimensions do not change. Moreover, 

$$\Square_i^k(R^e_{j,l}(z)) = \left[ \begin{array}{cc}
				m_e(c_k, c_k, a_k, a_k) & m_e(c_k,d_k,b_k, a_k)	\\			
				c_k& c_k
				\end{array}
				\right] $$ 
for each $k \in n_i$.
\end{lem}

\begin{proof}
Without loss, assume $i=2, j=1$ and $l=0$. The proof is the same as that of Lemma \ref{lem:comrot}, except that now 

$$\left[ \begin{array}{cc}
				z_{0,0,k}& z_{0,0,k+1}\\
				z_{0,n_1-1,k}&z_{0,n_1-1,k+1}
				\end{array}
				\right] =\left[ \begin{array}{cc}
				a_k& a_k+1\\
				c_k&c_{k+1}
				\end{array}
				\right] \in M(\theta_2, \theta_1)$$
for each $k \in n_2-2$. So we choose $\alpha, \beta, \gamma , \epsilon \in \mathcal{C}_{n_0, n_1, n_2}$ so that 

\begin{enumerate}
\item $\Square_2^k{\alpha'} = A_k = \left[ \begin{array}{cc} 
										c_k & c_k \\
										c_k & c_k \end{array} \right]$
\item $\Square_2^k{\beta'} = B_k = \left[ \begin{array}{cc}
											c_k & d_k\\
											a_k & b_k \end{array} \right]$
\item $\Square_2^k{\gamma'} = C_k = \left[ \begin{array}{cc}
												a_k & b_k \\
												a_k & b_k \end{array} \right]$
\item $\Square_2^k{\epsilon'} = D_k = \left[ \begin{array}{cc} 
												a_k & a_k \\
												c_k & c_k \end{array} \right]$.

\end{enumerate}
for each $k \in n_2-1$.

It is easy to see that $\alpha, \beta, \gamma, \epsilon \in \mathcal{C}_{n_0, n_1, n_2}(\theta_0, \theta_1, \theta_2)$. As before, set $R^e_{j,l}(z) = m_e(\alpha, \beta , \gamma, \epsilon)$. Now we see that $$\Square_2^k(R^e_{j,l}(z)) = \left[ \begin{array}{cc}
				m_e(c_k, c_k, a_k, a_k) & m_e(c_k,d_k,b_k, a_k)	\\			
				c_k& c_k
				\end{array}
				\right], $$
as desired.
\end{proof}

Shift rotations have the nice property of preserving supporting lines. Specifically, we mean that 

\begin{lem}\label{lem:suppreserved}
Let $\var$ be a modular variety with Day terms $m_0, \dots, m_n$. Take $\A \in \var$ and $(\theta_0, \theta_1, \theta_2) \in \Con(\A)^3$.
Suppose $z \in \mathcal{C}(\theta_0, \theta_1, \theta_2)$ is such that every $(j)$-supporting line is a $\delta$-pair, for some $j \in 3$ and $\delta \in \Con(\A)$. Then, for each $e \in n+1$, every $(l)$-supporting line of $R^e_{j,l}$ is a $\delta$-pair.

\end{lem}
\begin{proof}
This follows from Lemmas \ref{lem:comrot} and \ref{prop:shift}.

\end{proof} 

\begin{thm}\label{thm:transimregcen}
$C(\theta_0, \theta_1, \theta_2; \delta)$ implies $C_{tr}(\theta_0, \theta_1, \theta_2; j; \delta)$ for each coordinate $j \in 3$.

\end{thm}

\begin{proof}
Take $z \in \mathcal{C}_{n_0, n_1, n_2}(\theta_0, \theta_1, \theta_2)$ such that each $(j)$-supporting line of $z$ is a $\delta$-pair. We assume that $C(\theta_0, \theta_1, \theta_2;\delta)$ holds and will use this assumption to show that the $(j)$-pivot line of $z$ is also a $\delta$-pair with an induction on the dimensions $n_j$ and $n_l$ for $l \neq j$. 

The base case is when both of these dimensions are equal to $2$. Without loss, assume that $j=1$ and $l=0$. Then, for each $e \in n+1$ we have by Lemma \ref{lem:suppreserved} that every $(0)$-supporting line of $R^e_{1,0}(z)$ is a $\delta$-pair. 
By Lemma \ref{lem:rotwhen2}, $R^e_{1,0}(z) \in \mathcal{C}_{2,2,n_2}(\theta_0, \theta_1, \theta_2)$, that is, $\Square_2^k(R_{1,0}^e(z)) \in M(\theta_0, \theta_1)$ for each $k \in n_2$. Moreover, because the $(0)$-supporting line of each $\Square_2^k(R_{1,0}^e(z))$ is a $\delta$-pair we can use the assumption that $C(\theta_0, \theta_1, \theta_2; \delta)$ holds to conclude that the $(0)$-pivot line of $R_{1,0}^e$ is a $\delta$-pair. Because this holds for every $e \in n+1$, we apply Lemma \ref{prop:shift} to conclude that the $(1)$-pivot line of $z$ is a $\delta$-pair. This is shown in Figure \ref{fig:ttermbasecase}, where $\delta$-pairs are drawn as curved lines and constant pairs are drawn in bold. 

The result now follows from an induction on Lemma \ref{lem:comrot}.
\begin{figure}[ht]

\includegraphics{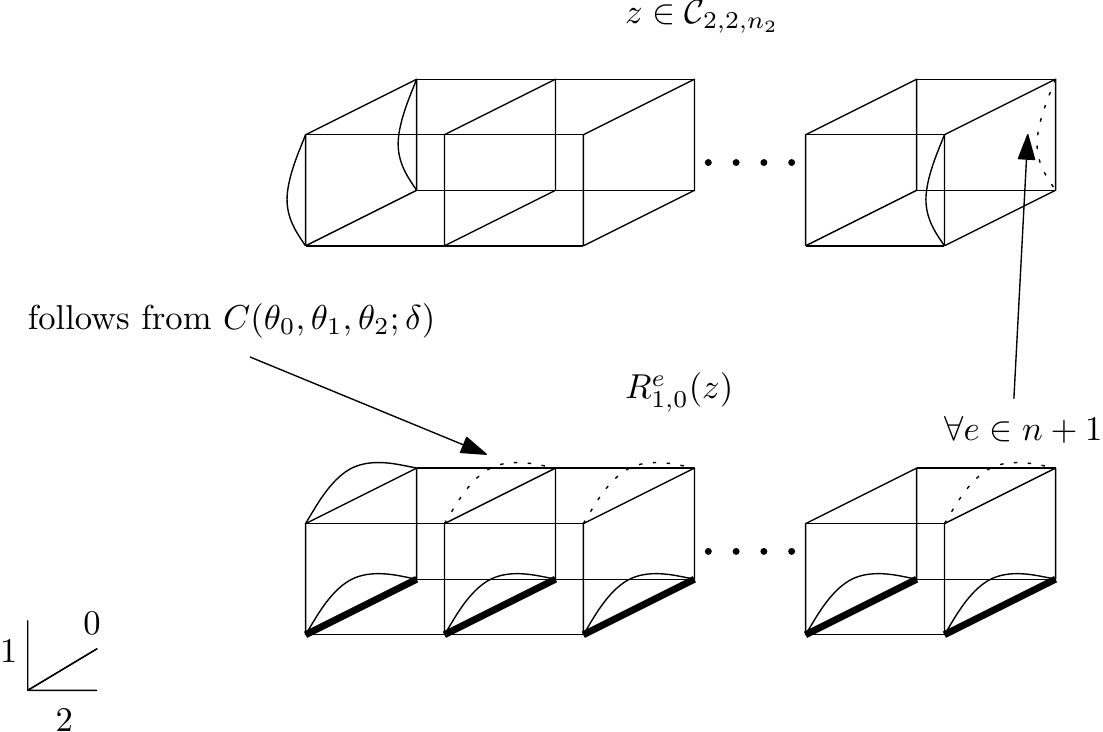}

\caption{}\label{fig:ttermbasecase}
\end{figure}

\end{proof}

\section{ Definition and Properties of $\Delta$}

We define $\Delta(\theta_0, \theta_1, \theta_2)$ in a manner analogous to the $2$-dimensional $\Delta$. We will prove that this definition is equivalent to a transitive closure of a relation that is obtained by looking at the corners of complexes. This perspective is useful for showing properties of $\Delta(\theta_0, \theta_1, \theta_2)$ because of the equivalence of the transitive term condition with the term condition.

Recall that $\Delta_{\theta_i} = \theta_i$, where we think of $\theta_i$ as being the collection of $1$-dimensional $\theta_i$-matrices. As before, $\Delta_{\theta_i, \theta_j} = \Cg^{\Delta_{\theta_i}}( \{ \cube^2_j(x,y): \langle x, y \rangle \in \theta_j \})$. Let $k \in 3$ such that $k\neq i,j$. The algebra $(\Delta_{\theta_{i}, \theta_{j}})^2$ is the second power of a square, so it is naturally coordinatized by $2^3$. That is, to each $\langle a,b \rangle \in (\Delta_{\theta_i, \theta_j})^2$ we associate a matrix $z \in A^{2^3}$ such that $\face_k^0 = a$ and $\face_k^1 = b$. We now set

$$\Delta_{\theta_i, \theta_j, \theta_k} = \Cg^{\Delta_{\theta_i, \theta_j}}( \{ \cube^3_k(x,y): \langle x, y \rangle \in \theta_k \}) \subset  A^{2^3}$$

Without loss, we will show that $\Delta_{\theta_0, \theta_2, \theta_1} = \Delta_{\theta_0, \theta_1, \theta_2}$. The result then follows from the symmetry of the $2$-dimensional $\Delta$.

We begin with the following 

\begin{lem}\label{lem:symdel1}
Suppose that $z_0, z_1 \in \Delta_{\theta_1, \theta_2, \theta_0}$ are such that 

$$ \Square_2^0(z_0) = \left[ \begin{array}{cc}
								d_1 & d_1 \\
								d_0 & d_0 \end{array} \right]\text{, }  \Square_2^1(z_0) = \Square_2^0(z_1)= \left[ \begin{array}{cc}
											a_1 & a_1 \\
											a_0 & a_0 \end{array} \right] \text{ and }$$

$ \Square_2^0(z_1) = \left[ \begin{array}{cc}
								b_1 & c_1 \\
								 b_0& c_0 \end{array} \right].$
								 
Then $ \left[ \begin{array}{cc}
				d_1 &b_1 \\
				d_0 & b_0 \end{array} \right] 
\Delta_{\theta_1, \theta_2, \theta_0}
\left[ \begin{array}{cc}
				d_1 &c_1 \\
				d_0 & c_0 \end{array} \right]. $

\end{lem}

\begin{proof}
We provide in Figure \ref{fig:symdelt_1} a picture of the two matrices $z_0, z_1 \in \Delta_{\theta_1, \theta_2, \theta_0}.$ 

\begin{figure}[ht]

\includegraphics{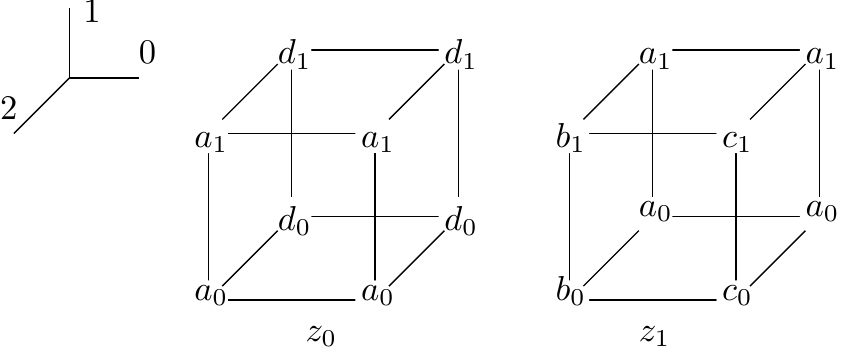}

\caption{}\label{fig:symdelt_1}
\end{figure}

Set 

\begin{enumerate}
\item
$a = \left[ \begin{array}{cc}
d_1 &b_1 \\
d_0 & b_0 \end{array} \right] $ 
\item
				
$b= \left[ \begin{array}{cc}
a_1 &b_1 \\
a_0 & b_0 \end{array} \right] $ 
\item
$c= \left[ \begin{array}{cc}
				d_1 &c_1 \\
				d_0 & c_0 \end{array} \right]$ 
\item		
$d=\left[ \begin{array}{cc}
				a_1 &c_1 \\
				a_0 & c_0 \end{array} \right]$
				
\end{enumerate}

Now, we want to show that 

\[\langle a, b \rangle \in \Delta_{\theta_1, \theta_2, \theta_0}.
\] 
By Lemma \ref{prop:shift} it suffices that $\langle m_e(a,a,c,c) , m_e(a,b,c,d) \rangle \in \Delta_{\theta_1, \theta_2, \theta_0}$ for all $e \in n+1$. This amounts to checking that 

$$ \left[ \begin{array}{cc}
			m_e(d_1, d_1, d_1, d_1) &  m_e(b_1, b_1, c_1, c_1) \\
			m_e(d_0, d_0, d_0, d_0) & m_e(b_0, b_0, c_0, c_0) \end{array} \right] 
			\Delta_{\theta_1, \theta_2, \theta_0} 
			\left[ \begin{array}{cc}
			m_e(d_1, a_1, a_1, d_1) &  m_e(b_1, b_1, c_1, c_1) \\
			m_e(d_0, a_0, a_0, d_0) & m_e(b_0, b_0, c_0, c_0) \end{array} \right], 
			$$
which follows easily from the identity $m_e(x,y,y,x) \approx x$.
\end{proof}

\begin{lem}\label{lem:symdel2}
Suppose that $z_0, z_1 \in \Delta_{\theta_1, \theta_2, \theta_0}$ are such that 

$$ \Square_2^0(z_0) = \left[ \begin{array}{cc}
								x_1 & y_1 \\
								x_0 & y_0 \end{array} \right] \text{, }   \Square_2^1(z_0) = \Square_2^0(z_1)= \left[ \begin{array}{cc}
											u_1 & v_1 \\
											u_0 & v_0 \end{array} \right], $$
and 
$ \Square_2^0(z_1) = \left[ \begin{array}{cc}
								r_1 & s_1 \\
								 r_0& s_0 \end{array} \right] .$
								 
Then $ \left[ \begin{array}{cc}
				x_1 &r_1 \\
				x_0 & r_0 \end{array} \right] 
\Delta_{\theta_1, \theta_2, \theta_0}
\left[ \begin{array}{cc}
				y_1 &s_1 \\
				y_0 & s_0 \end{array} \right] $

\end{lem}

\begin{proof}
We provide in Figure \ref{fig:symdelt_3} a picture of the two matrices $z_0, z_1 \in \Delta_{\theta_1, \theta_2, \theta_0}.$ 

\begin{figure}[ht]

\includegraphics{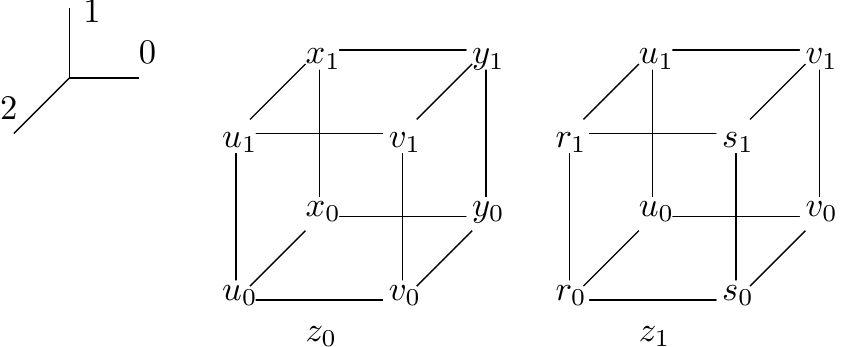}

\caption{}\label{fig:symdelt_3}
\end{figure}

Set 

\begin{enumerate}
\item
$a = \left[ \begin{array}{cc}
x_1 &r_1 \\
x_0 & r_0 \end{array} \right] $ 
\item
				
$b= \left[ \begin{array}{cc}
u_1 &r_1 \\
u_0 & r_0 \end{array} \right] $ 
\item
$c= \left[ \begin{array}{cc}
				y_1 &s_1 \\
				y_0 & s_0 \end{array} \right]$ 
\item		
$d=\left[ \begin{array}{cc}
				v_1 &s_1 \\
				v_0 & s_0 \end{array} \right]$
				
\end{enumerate}

We want to show that $\langle a, b \rangle \in \Delta_{\theta_1, \theta_2, \theta_0}$. By Lemma \ref{prop:shift} it suffices that $\langle m_e(a,a,c,c) , m_e(a,b,c,d) \rangle \in \Delta_{\theta_1, \theta_2, \theta_0}$ for all $e \in n+1$. An examination of $z_1$ reveals that
\begin{enumerate}
\item
$\left[\begin{array}{cc}
			v_1 & s_1\\
			v_0 & s_0 \end{array} \right]
			 \Delta_{\theta_1, \theta_2, \theta_0} 
			\left[\begin{array}{cc}
			u_1 & r_1\\
			u_0 & r_0 \end{array} \right]
		$
\smallskip
\item

$\left[\begin{array}{cc}
			v_1 & v_1\\
			v_0 & v_0 \end{array} \right]
			 \Delta_{\theta_1, \theta_2, \theta_0} 
			\left[\begin{array}{cc}
			u_1 & u_1\\
			u_0 & u_0 \end{array} \right]
		$
\end{enumerate}

Now, because for all $e \in n+1$ the identity $m_e(x,y,y,x) \approx x$ holds, it follows that

\begin{enumerate}
\item
$m_e\left(\left[\begin{array}{cc}
			u_1 & r_1\\
			u_0 & r_0 \end{array} \right]
			,
			\left[\begin{array}{cc}
			u_1 & r_1\\
			u_0 & r_0 \end{array} \right]
			,
			\left[\begin{array}{cc}
			v_1 & s_1\\
			v_0 & s_0 \end{array} \right]
			,
			\left[\begin{array}{cc}
			v_1 & s_1\\
			v_0 & s_0 \end{array} \right]
			\right) \Delta_{\theta_1, \theta_2, \theta_0} 
			\left[\begin{array}{cc}
			u_1 & r_1\\
			u_0 & r_0 \end{array} \right]
		$
\smallskip
\item

$m_e\left(\left[\begin{array}{cc}
			u_1 & r_1\\
			u_0 & r_0 \end{array} \right]
			,
			\left[\begin{array}{cc}
			u_1 & u_1\\
			u_0 & u_0 \end{array} \right]
			,
			\left[\begin{array}{cc}
			v_1 & v_1\\
			v_0 & v_0 \end{array} \right]
			,
			\left[\begin{array}{cc}
			v_1 & s_1\\
			v_0 & s_0 \end{array} \right]
			\right) \Delta_{\theta_1, \theta_2, \theta_0} 
			\left[\begin{array}{cc}
			u_1 & r_1\\
			u_0 & r_0 \end{array} \right]
		$
\end{enumerate}
Therefore, the matrix shown in Figure \ref{fig:symdelt_4} belongs to $\Delta_{\theta_1, \theta_2, \theta_0}$ for all $e \in n+1$.

\begin{figure}[ht]
\includegraphics{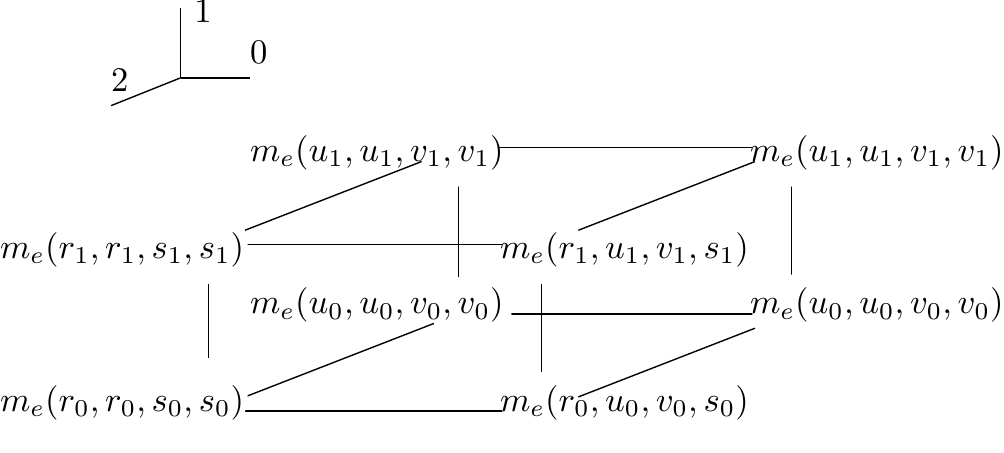}
\caption{}\label{fig:symdelt_4}
\end{figure}

An examination of $z_0$ reveals that 
\begin{enumerate}
\item
$\left[\begin{array}{cc}
			x_1 & u_1\\
			x_0 & u_0 \end{array} \right]
			 \Delta_{\theta_1, \theta_2, \theta_0} 
			\left[\begin{array}{cc}
			x_1 & u_1\\
			x_0 & u_0  \end{array} \right]
		$
\smallskip
\item
$\left[\begin{array}{cc}
			y_1 & v_1\\
			y_0 & v_0 \end{array} \right]
			 \Delta_{\theta_1, \theta_2, \theta_0} 
			\left[\begin{array}{cc}
			y_1 & v_1\\
			y_0 & v_0 \end{array} \right]
		$
\end{enumerate}
We therefore obtain that the matrix depicted in Figure \ref{fig:symdelt_5} belongs to $\Delta_{\theta_1, \theta_2, \theta_0}$. An application of Lemma \ref{lem:symdel1} finishes the proof.
\begin{figure}[!ht]
\includegraphics{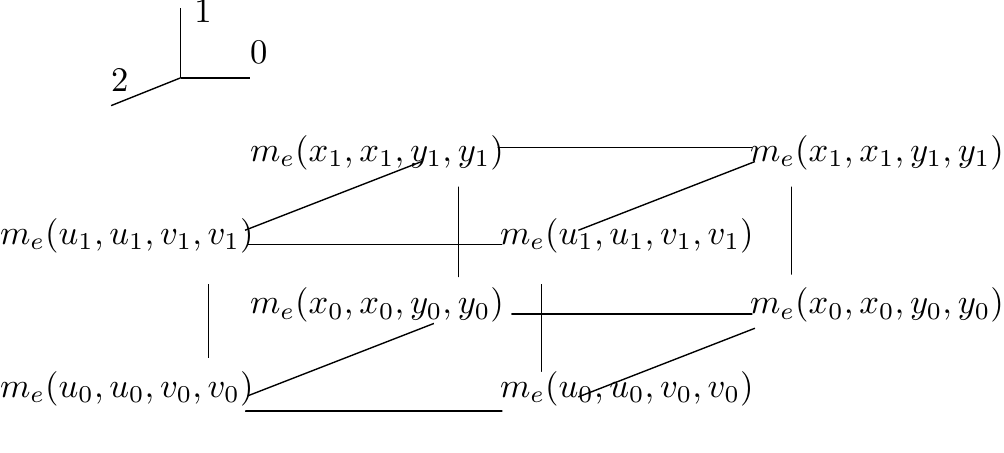}
\caption{}\label{fig:symdelt_5}
\end{figure}

\end{proof}

We now provide another characterization of $\Delta_{\theta_0, \theta_1, \theta_2}$. For $i \in 3$ set 

$$R_i = \{ \langle \Square_i^0(\Corner(z)) , \Square_i^1(\Corner(z)) \rangle : z \in \mathcal{C}_{2,2, n_2}(\theta_0, \theta_1, \theta_2) \text{ for some } n_2 \geq 2\}$$

We will often abuse notation and refer to elements of $R_i$ as elements of $ A^{2^3}$, rather than pairs of elements of $A^{2^2}$. A picture of a complex that corresponds to the pair 
$$\bigg\langle \left[ \begin{array}{cc}
					c&g\\
					a&e \end{array} \right], 
					\left[ \begin{array}{cc}
					d&h\\
					b&f \end{array} \right] \bigg\rangle  \in R_0$$
is given in Figure \ref{fig:symdelt_6}.
\begin{figure}[!ht]
\includegraphics{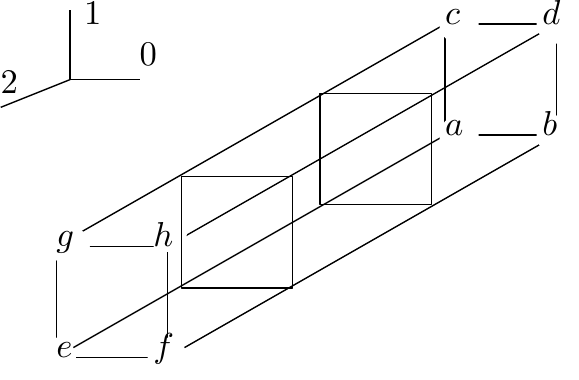}
\caption{}\label{fig:symdelt_6}
\end{figure}

\begin{lem}\label{lem:symdel4}
$\Delta_{\theta_i, \theta_j, \theta_l}$ is the transitive closure of $R_l$.

\end{lem}

\begin{proof}
Without loss, we show that $\Delta_{\theta_1, \theta_2, \theta_0}$ is the transitive closure of $R_0$.

Notice that $R_0 \in (\Delta_{\theta_1, \theta_2})^2$. It is easy to see that $R_0$ is reflexive, symmetric and compatible. Therefore, $(R_0)^{TC} \in \Con(\Delta_{\theta_1, \theta_2}).$ Notice that $M(\theta_0, \theta_1, \theta_2) \subseteq \Delta_{\theta_1, \theta_2, \theta_0}$, because each generator of $M(\theta_0, \theta_1, \theta_2)$ belongs to $\Delta_{\theta_1, \theta_2, \theta_0}$. Therefore, $R_0 \subseteq \Delta_{\theta_1, \theta_2, \theta_0}$ by Lemma \ref{lem:symdel2}, so $(R_0)^{TC} = \Delta_{\theta_1, \theta_2, \theta_0}$.

\end{proof}

A similar argument shows that $R_2 \in \Delta_{\theta_1, \theta_2, \theta_0}$. An application of Lemma \ref{lem:symdel2} then gives that $\Delta_{\theta_0, \theta_1, \theta_2} \subseteq \Delta_{\theta_1, \theta_2, \theta_0}$. The argument is clearly symmetric, so we have proved the following 

\begin{thm}\label{thm:symdel1}
$\Delta_{\theta_i, \theta_j, \theta_k} = \Delta_{\theta_{\pi(i)}, \theta_{\pi(j)}, \theta_{\pi(k)}}$ for any permutation of coordinates $\pi$.
\end{thm}

We also have that 

\begin{thm}\label{thm:tridelisbindel}
$\Delta_{\theta_i, \theta_j, \theta_l} = \Delta_{\Delta_{\theta_i, \theta_j}, \Delta{\theta_i, \theta_l}}$
\end{thm}

\begin{proof}
Without loss, we show that $\Delta_{\theta_1, \theta_2, \theta_0} = \Delta_{\Delta_{\theta_1, \theta_2}, \Delta_{\theta_1, \theta_0}}$.

Both of these relations can be considered as congruences on $\Delta_{\theta_1, \theta_2}$. Each one contains the generators of the other, so they are equal.
\end{proof}

We now use the $\Delta$ relation to characterize the ternary commutator.

\begin{thm}\label{thm:maindeltathm}
The following are equivalent:

\begin{enumerate}
\item
$\langle x, y \rangle \in [\theta_0, \theta_1, \theta_2]$
\item
$ \bigg\langle \left[ \begin{array}{cc}
						x & x\\
						x & x \end{array} \right], 
						\left[ \begin{array}{cc}
						x & y\\
						x & x \end{array} \right] \bigg\rangle \in \Delta_{\theta_0, \theta_1, \theta_2}$
\smallskip				
\item
$ \bigg\langle \left[ \begin{array}{cc}
						a_0 & a_1\\
						a_0 & a_1 \end{array} \right], 
						\left[ \begin{array}{cc}
						a_2 & y\\
						a_2 & x \end{array} \right] \bigg\rangle \in \Delta_{\theta_0, \theta_1, \theta_2}$, for some $a_0, a_1, a_2 \in \A$ \item

\smallskip
$ \bigg\langle \left[ \begin{array}{cc}
						b_1& b_1\\
						b_0& b_0 \end{array} \right], 
						\left[ \begin{array}{cc}
						b_2 & b_2\\
						y & x \end{array} \right] \bigg\rangle \in \Delta_{\theta_0, \theta_1, \theta_2}$ for some $b_0, b_1, b_2\in \A$\item
\smallskip
$ \bigg\langle \left[ \begin{array}{cc}
						c_1 & y\\
						c_0 & c_2 \end{array} \right], 
						\left[ \begin{array}{cc}
						c_1 &x \\
						c_0 & c_2 \end{array} \right] \bigg\rangle \in \Delta_{\theta_0, \theta_1, \theta_2}$

\end{enumerate}

\end{thm}

\begin{proof}
That $(2)$ implies $(3),(4),(5)$ is trivial. Suppose that $(2)$ is true for $x,y \in \A$. We show that $(1)$ holds. Because $\Delta_{\theta_0, \theta_1, \theta_2} = (R_2)^{TC}$, there is a sequence of complexes $z_0, \dots, z_{s-1} \in \mathcal{C}_{n_0 \geq 2, 2,2}(\theta_0, \theta_1, \theta_2)$ such that
\begin{enumerate}
\item
$\Square_2^0({z_0}) = \left[ \begin{array}{cc}
						x & x\\
						x & x \end{array} \right]$
\item
$\Square_2^0({z_{s-1}}) = \left[ \begin{array}{cc}
						x & y\\
						x & x \end{array} \right]$
						
\item
$\Square_2^1(z_k) = \Square_2^0(z_{k+1}) = \left[ \begin{array}{cc}
						c_k & d_k\\
						a_k & b_k \end{array} \right]$ for each $k \in s-1$.

\end{enumerate}
This situation is depicted in Figure \ref{fig:delandcomm_1}, where $[\theta_0, \theta_1, \theta_2]$-pairs are connected by a solid curved line. 

\begin{figure}[!ht]
\includegraphics{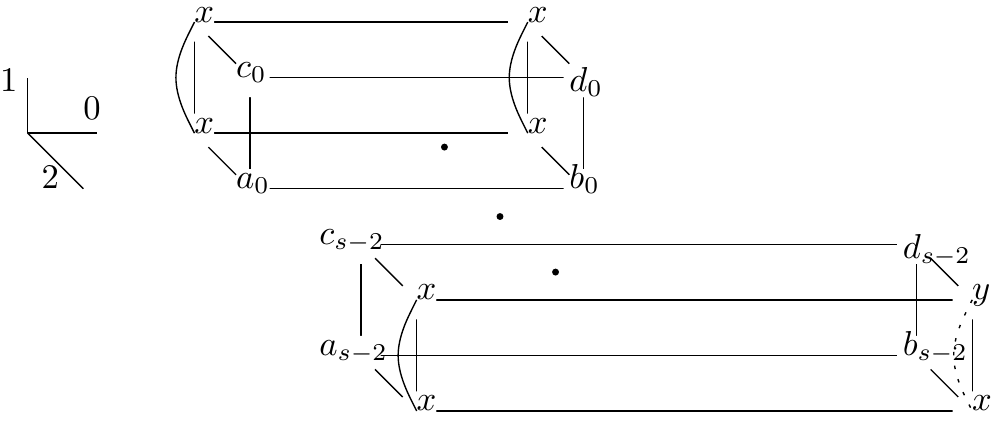}
\caption{}\label{fig:delandcomm_1}
\end{figure}
We wish to show that the pair connected by a dashed curved line is also a $[\theta_0, \theta_1, \theta_2]$-pair. By Lemma \ref{prop:shift}, it suffices to see that $\langle m_e(x,x,x,x), m_e(x,y,x,x) \rangle \in [\theta_0, \theta_1, \theta_2]$ for each $e \in n+1$.
Now, for $e \in n+1$ we have for each $z_k$ the $e$th shift rotation at $(0,1)$. By Lemma \ref{lem:comrot}, these complexes have the following faces:

\begin{enumerate}
\item
$\Square_2^0(R^e_{0,1}({z_0})) = \left[ \begin{array}{cc}
						x & x\\
						x & x \end{array} \right]$
\item
$\Square_2^0(R^e_{0,1}({z_{s-1}})) = \left[ \begin{array}{cc}
						m_e(x,x,x,x) & m_e(x,y,x,x) \\
						x & x \end{array} \right]$
						
\item
$\Square_2^1(R^e_{0,1}(z_k)) = \Square_2^0(z_{k+1}) = \left[ \begin{array}{cc}
						m_e(c_k, c_k, a_k, a_k) & m_e(c_k, d_k, b_k, a_k) \\
						c_k & c_k \end{array} \right]$ for each $k \in s-1$.

\end{enumerate}

Now, $C(\theta_0, \theta_1, \theta_2; [\theta_0, \theta_1, \theta_2])$ holds, which implies that $C_{tr}(\theta_0, \theta_1, \theta_2; [\theta_0, \theta_1, \theta_2])$ holds. So, $\langle m_e(x,x,x,x), m_e(x,y,x,x) \rangle \in [\theta_0, \theta_1, \theta_2]$ by induction. This is shown in Figure, where constant pairs are drawn in bold.

\begin{figure}[!ht]
\includegraphics{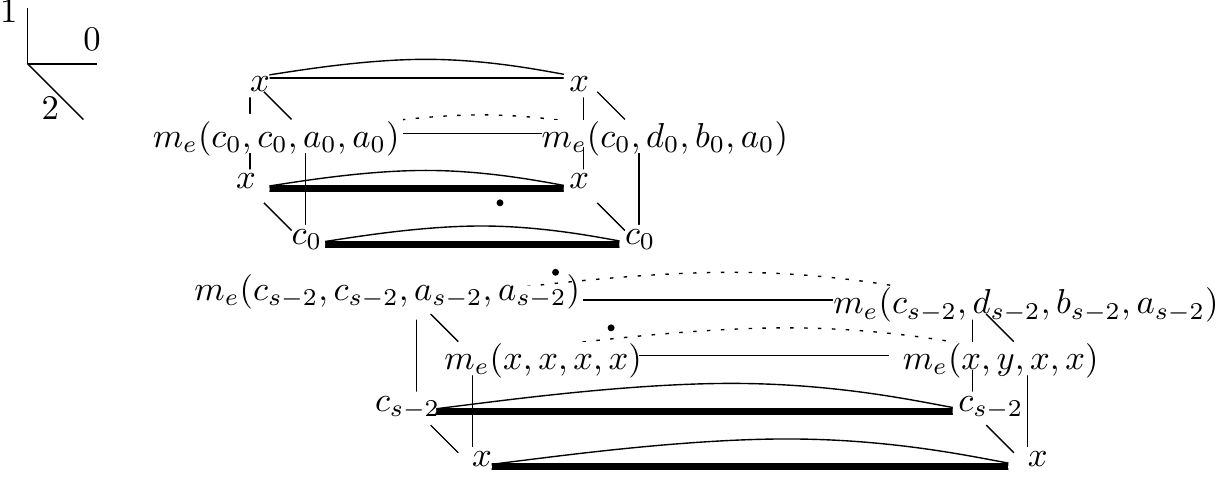}
\caption{}\label{fig:delandcomm_2}
\end{figure}

A similar argument shows that $(3),(4), (5)$ all imply $(1)$. To finish we show that $(1)$ implies $(2)$. For this it suffices to show that the set of pairs $\langle x, y \rangle \in \A^2$ such that 
$$ \bigg\langle \left[ \begin{array}{cc}
						x & x\\
						x & x \end{array} \right], 
						\left[ \begin{array}{cc}
						x & y\\
						x & x \end{array} \right] \bigg\rangle \in \Delta_{\theta_0, \theta_1, \theta_2}$$
is a congruence of $\A$. Reflexivity and compatibility are obvious. The proof of symmetry is shown in Figure \ref{fig:delandcomm_3}. The proof of transitivity is similar to that of symmetry.

\begin{figure}[!ht]
\includegraphics[scale=.9]{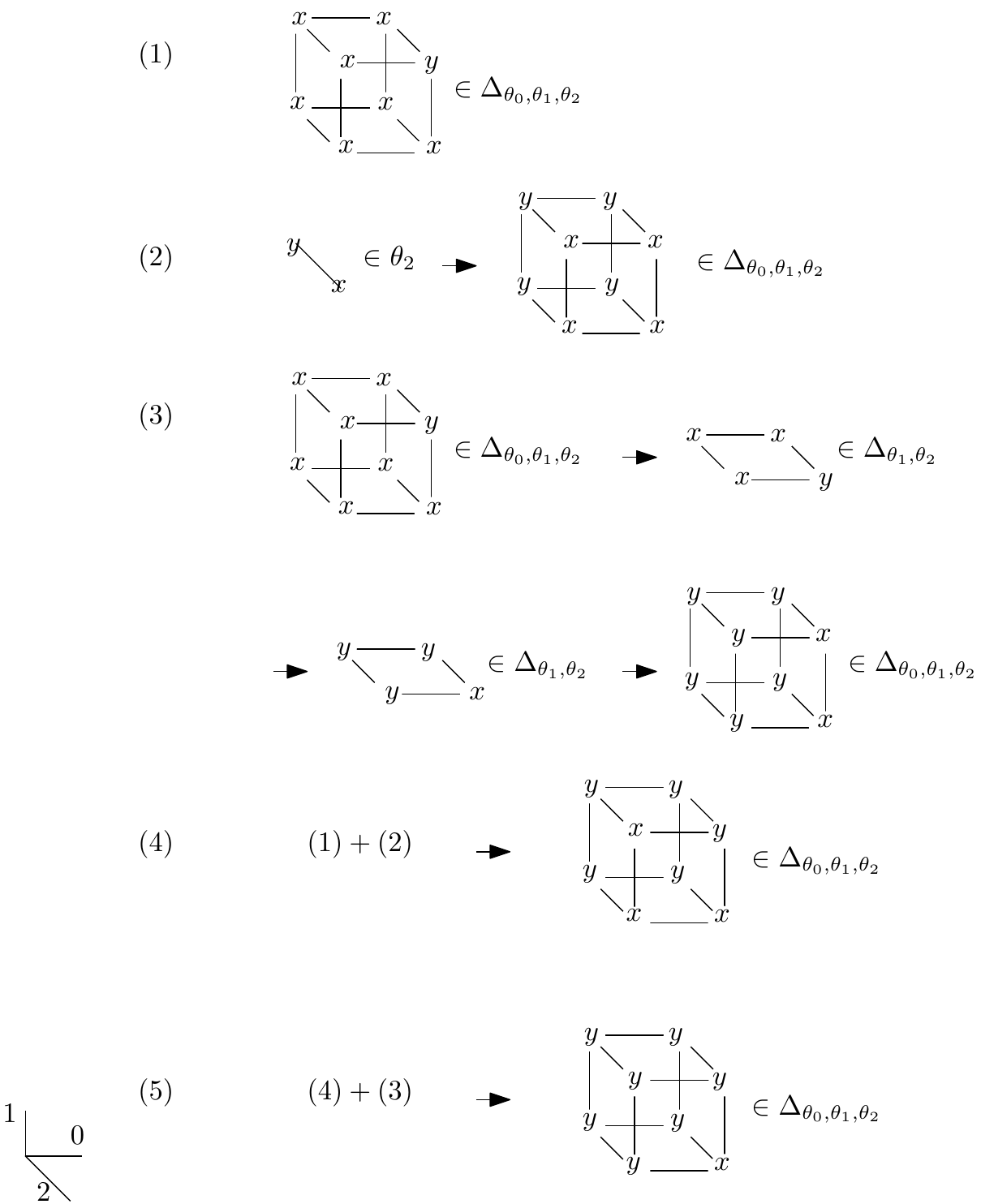}
\caption{}\label{fig:delandcomm_3}
\end{figure}
\end{proof}

\newpage

\section{Greatest Ternary Congruence Lattice Operation}

A beautiful result in the theory of the binary commutator is the following from \cite{FM}.

\begin{thm}

Let $\var$ be a modular variety, and suppose that for each $\A \in \var$ there is a binary operation $C: \Con(\A)^2 \rightarrow \Con(\A)$ such that, for all $\A \in \var$ and $\theta_0, \theta_1 \in \Con(\A)$,

\begin{enumerate}
\item $C(\theta_0, \theta_1)  \leq \theta_0 \meet \theta_1$ 
\item If $f$ is a surjective homomorphism with kernel $\pi$, then $$C(\theta_0, \theta_1) \join \pi = f^{-1}C(f(\theta_0 \join \pi, f(\theta_0 \join \pi)).$$ 

\end{enumerate}
Then $C(\theta_0, \theta_1) \leq [\theta_0, \theta_1]$ for all $A \in \var$ and $\theta_0, \theta_1 \in \Con(\A)$.

\end{thm}

Interestingly, an analogous property holds for the ternary commutator. 

\begin{prop}\label{global}

Let $\var$ be a congruence modular variety. Suppose that $C:\Con(\A)^3 \rightarrow \Con(\A)$ is defined for each $\A \in \var$ such that 
\begin{enumerate}
\item $C(\theta_0, \theta_1, \theta_2) \leq \theta_0 \meet \theta_1 \meet \theta_2$
\item $C (\theta_0, \theta_1, \theta_2) \join \pi = f^{-1}(C(f(\theta_0 \join \pi),f(\theta_1 \join \pi),f(\theta_2 \join \pi)))$ for $f$ a surjective homomorphism with kernel $\pi$. 

\end{enumerate}

Then $C(\theta_0, \theta_1, \theta_2) \leq [\theta_0, \theta_1, \theta_2]$.
\end{prop}
\begin{proof}

Take $\theta_0, \theta_1, \theta_2 \in \Con(\A)$ for some $\A \in \var$. 
Consider the following congruences of $\Delta_{\theta_0, \theta_1}$:
\begin{align*}
\Delta &= \Delta_{\theta_0, \theta_1, \theta_2} \\
[\theta_0, \theta_1, \theta_2]_0 &= \bigg\{ \bigg\langle\left[\begin{array}{cc}
													u_0& x\\
													v_0 & w_0 \end{array}\right], 
													\left[\begin{array}{cc}
													u_1& y\\
													v_1 & w_1 \end{array}\right]\bigg\rangle :
\langle x,y \rangle \in [\theta_0, \theta_1, \theta_2] \bigg\}\\
\eta_0 &= \bigg\{ \bigg\langle\left[\begin{array}{cc}
													u_0& x\\
													v_0 & w_0 \end{array}\right], 
													\left[\begin{array}{cc}
													u_0& y\\
													v_0 & w_1 \end{array}\right]\bigg\rangle :
\langle x,y \rangle \in \theta_0 \bigg\}\\
\eta_0 &= \bigg\{ \bigg\langle\left[\begin{array}{cc}
													u_0& x\\
													v_0 & w_0 \end{array}\right], 
													\left[\begin{array}{cc}
													u_1& y\\
													v_0 & w_0 \end{array}\right]\bigg\rangle :
\langle x,y \rangle \in \theta_1 \bigg\}\\
\end{align*}

Let $f:\Delta_{\theta_0, \theta_1} \rightarrow \A$ be the projection of $\Delta_{\theta_0, \theta_1}$ onto the top right vertex and let $\pi$ be the kernel of $f$. One can check that $f([\theta_0, \theta_1, \theta_2]_0) = [\theta_0, \theta_1, \theta_2]$, $f(\Delta \join \pi) = \theta_2$, $f(\eta_0 \join \pi) = \theta_0$ and $f(\eta_1 \join \pi) = \theta_1$.

We compute $$C(\eta_0, \eta_1, \Delta) \leq \eta_0 \meet \eta_1 \meet \Delta \leq [\theta_0, \theta_1, \theta_2]_0$$
where the first inequality follows from the definition of $C$, and the second from Theorem \ref{thm:maindeltathm}. We then have that 

\begin{align*}
C(\theta_0, \theta_1, \theta_2) &= C(f(\eta_0 \join \pi),f(\eta_1 \join \pi), f(\Delta \join \pi))\\
&= f(C(\eta_0, \eta_1, \Delta)  \join \pi)\\
&\leq [\theta_0, \theta_1, \theta_2]
\end{align*}

\end{proof}

An immediate consequence of Proposition \ref{global} is the following

\begin{thm}
For $\theta_0, \theta_1, \theta_2 \in \Con(\A)$, we have that

 $$[[\theta_0, \theta_1], \theta_2] \leq [\theta_0, \theta_1,\theta_2].$$

\end{thm}

\section{Ternary Commutator and 3-Dimensional Cube Terms}

In \cite{oprsal}, Oprsal defines what is called a \textbf{strong cube term} for a congruence permutable variety. These strong cube terms are definable from a Mal'cev operation. We reproduce this definition for the $3$-dimensional case. Let $h(x_{(0,0)}, x_{(1,0)}, x_{(0,1)})$ be a term operation satisfying the identities

\begin{enumerate}
\item $h(x,x,y) \approx y$

\item $h(x,y,x) \approx y.$

\end{enumerate}
Notice that $h$ satisfies the Mal'cev identities up to a permutation of variables. We now define 
$$ p(x_{(0,0,0)},x_{(1,0,0)},x_{(0,1,0)},x_{(1,1,0)},x_{(0,0,1)},x_{(1,0,1)},x_{(0,1,1)}):= $$
$$h\big(h(x_{(0,0,0)},x_{(1,0,0)},x_{(0,1,0)}),x_{(1,1,0)}, h(x_{(0,0,1)},x_{(1,0,1)},x_{(0,1,1)})\big)$$

The operation $p$ is called a $3$-dimensional strong cube term. It satisfies the identities

\begin{enumerate}
\item $p(w,w,x,x,y,y,z) \approx z$
\item $p(w,x,w,x,y,z,y) \approx z$
\item $p(w,x,y,z,w,x,y) \approx z$
\end{enumerate}

Now, every modular variety has a weakened Mal'cev operation called a difference term. This is a term operation $d(x_{(0,0)}, x_{(1,0)}, x_{(0,1)})$ that satisfies 
\begin{enumerate}
\item $d(x,x,y) \approx y$
\item $d(x,y,x) \equiv_{[\theta, \theta]} y$, where $\theta = \Cg(\langle x,y \rangle ).$
\end{enumerate}
Notice the permutation of variables. We will later need the following

\begin{prop}[Special Case of 5.7 in \cite{FM}] \label{prop:binarydeldifferenceterm}
Let $\theta \in \Con(\A)$, and suppose $x \equiv_\theta y \equiv_\theta z$. Then 

$$\left[\begin{array}{cc}
			z & d(x,y,z)\\
			x & y \end{array} \right] \in \Delta_{\theta, \theta}$$
\end{prop}
We use the difference term to define a term operation that is analogous to a strong cube term. Set

$$ q(x_{(0,0,0)},x_{(1,0,0)},x_{(0,1,0)},x_{(1,1,0)},x_{(0,0,1)},x_{(1,0,1)},x_{(0,1,1)}):= $$
$$d\big(d(x_{(0,0,0)},x_{(1,0,0)},x_{(0,1,0)}),x_{(1,1,0)}, d(x_{(0,0,1)},x_{(1,0,1)},x_{(0,1,1)})\big)$$
We call $q$ a \textbf{$3$-dimensional wobbly cube term}. In fact, $q$ satisfies 

\begin{enumerate}
\item $q(x,x,x,x,y,y,y) \approx y$
\item $q(x,x,y,y,x,x,y) \approx y$
\item $q(x,y,x,y,x,y,x) \equiv_{[\theta, \theta, \theta]} y$, where $\theta = \Cg(\langle x,y \rangle )$,

\end{enumerate}
which we shall now demonstrate.
\begin{thm}\label{thm:wobblycompletesmatrices}
Let $\theta \in \Con(\A)$. Let $m\in A^{2^3}$ be such that $\face_0^0(m) , \face_1^0(m),\face_2^0(m) \in \Delta_{\theta, \theta}$. Set

\begin{align*}
(*) &= q(m_{(0,0,0)}, m_{(1,0,0)}, m_{(0,1,0)}, m_{(1,1,0)}, m_{(0,0,1)}, m_{(1,0,1)}, m_{(0,1,1)})\\
&= q(a,b,c,d,e,f,g)
\end{align*}
Take $z \in  A^{2^3}$ such that $z_f = m_f$ when $f \neq (1,1,1)$, and $z_{(1,1,1)} = (*)$. Then $z \in \Delta_{\theta_0, \theta_1, \theta_2}$.
\end{thm}

\begin{proof}
See Figure \ref{fig:wobblycube_1} for a picture of what is to be shown.

\begin{figure}[!ht]
\includegraphics{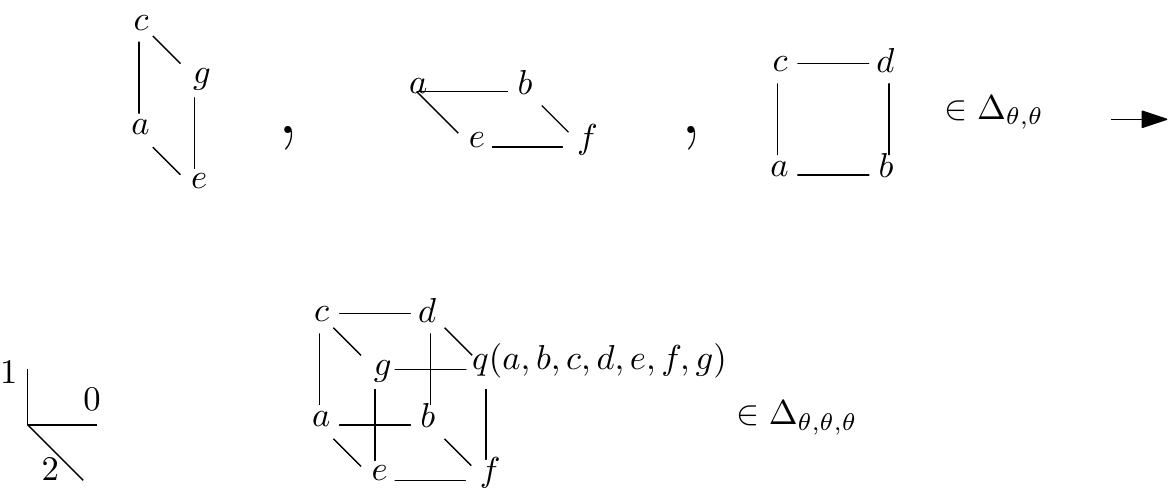}
\caption{}\label{fig:wobblycube_1}
\end{figure}

Let $A, B, C, D, E, F, G \in \Delta_{\theta_0, \theta_1, \theta_2}$ be the matrices in Figure \ref{fig:wobblycube_2}. As shown, \newline $q(A, B, C, D, E, F, G) \in \Delta_{\theta_0, \theta_1, \theta_2}$ also.

Suppose now that the conditions $(1), (2)$ and $(3)$ that are shown in Figure \ref{fig:wobblycube_3} hold. Using that $\Delta_{\theta, \theta, \theta}$ is a congruence that relates any two opposing faces will finish the proof. So, it remains to check that the conditions $(1), (2)$ and $(3)$ hold. We will show $(2)$. The proof of the others is similar.

Now, an application of Proposition \ref{prop:binarydeldifferenceterm} gives that
$ \left[ \begin{array}{cc}
			c & d(a,b,c)\\
			a & b \end{array} \right] \in \Delta_{\theta, \theta} $
so 
$$ \left[ \begin{array}{cc}
			c & d(a,b,c)\\
			a & b \end{array} \right] \Delta_{\theta, \theta, \theta}\left[ \begin{array}{cc}
			c & d(a,b,c)\\
			a & b \end{array} \right]. $$
Moreover, $ \left[ \begin{array}{cc}
			d & d(a,b,c)\\
			c & c \end{array} \right] \in \Delta_{\theta, \theta} $, because $ \left[ \begin{array}{cc}
			 c& d\\
			a & b \end{array} \right] \in \Delta_{\theta, \theta} $. Also, $ \left[ \begin{array}{cc}
			d(a,b,c) &d(a,b,c) \\
			c & c\end{array} \right] \in \Delta_{\theta, \theta} $.
			
Therefore, $\left[\begin{array}{c}
		d\\
		c \end{array}\right] \Delta_{\theta, \theta}\left[\begin{array}{c}
		d(a,b,c)\\
		c \end{array}\right]\Delta_{\theta, \theta} \left[\begin{array}{c}
		d(a,b,c)\\
		c \end{array}\right]$. By Theorem \ref{thm:tridelisbindel}, $\Delta_{\theta, \theta, \theta} = \Delta_{\Delta_{\theta, \theta}, \Delta_{\theta, \theta}}$, so we use Proposition \ref{prop:binarydeldifferenceterm} to conclude that 
		
		$$ \left[ \begin{array}{cc}
			d & d(a,b,c)\\
			c & c \end{array} \right] \Delta_{\theta, \theta, \theta}\left[ \begin{array}{cc}
			q(a,b,c,d,a,b,c) & d(a,b,c)\\
			c & c \end{array} \right]. $$
The result now follows, see Figure \ref{fig:wobblycube_4}.
\end{proof}

\begin{figure}
\begin{center}
\includegraphics[scale=1]{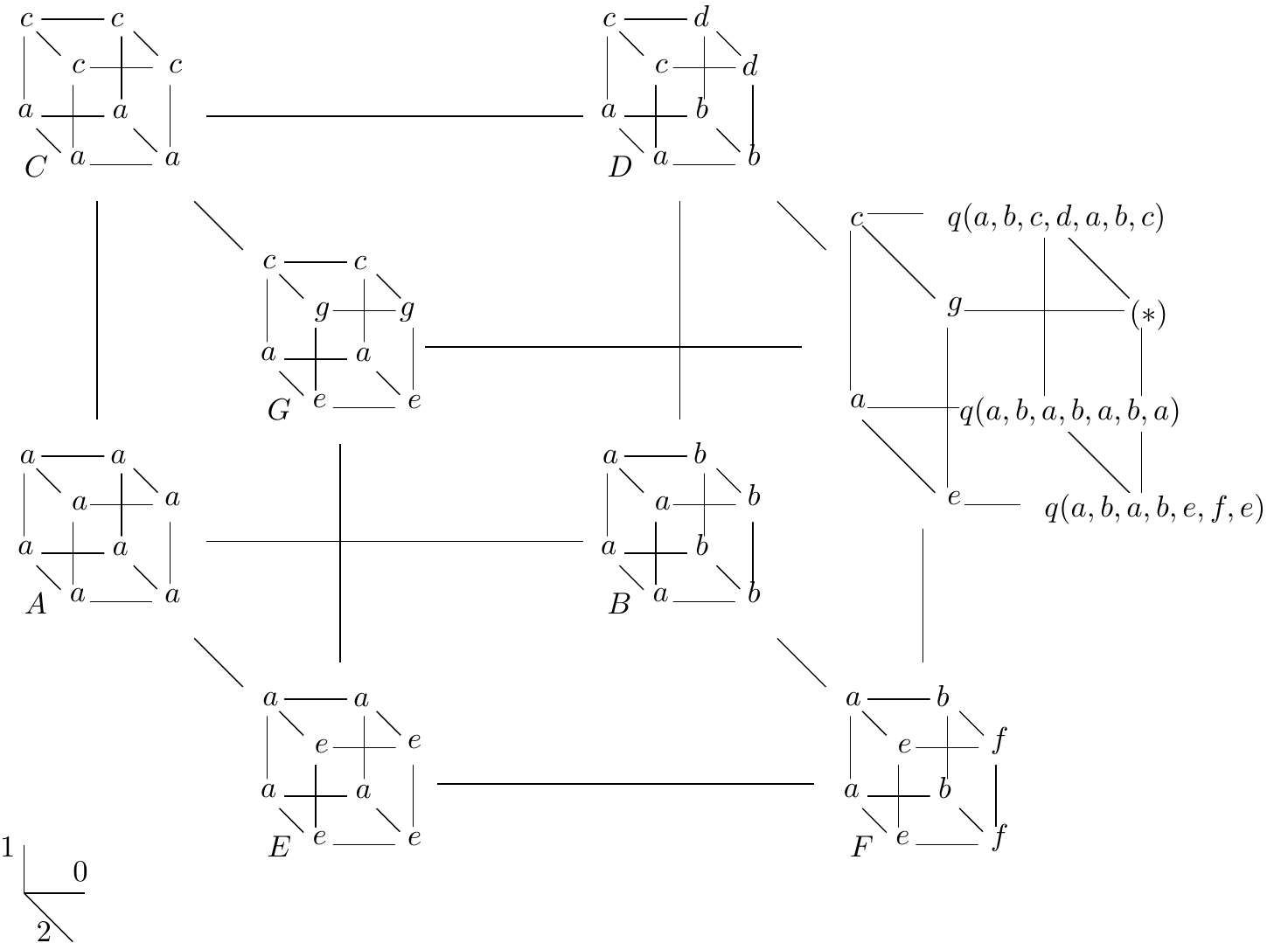}
\caption{}\label{fig:wobblycube_2}
\end{center}
\end{figure}

\begin{figure}
\includegraphics[scale=.9]{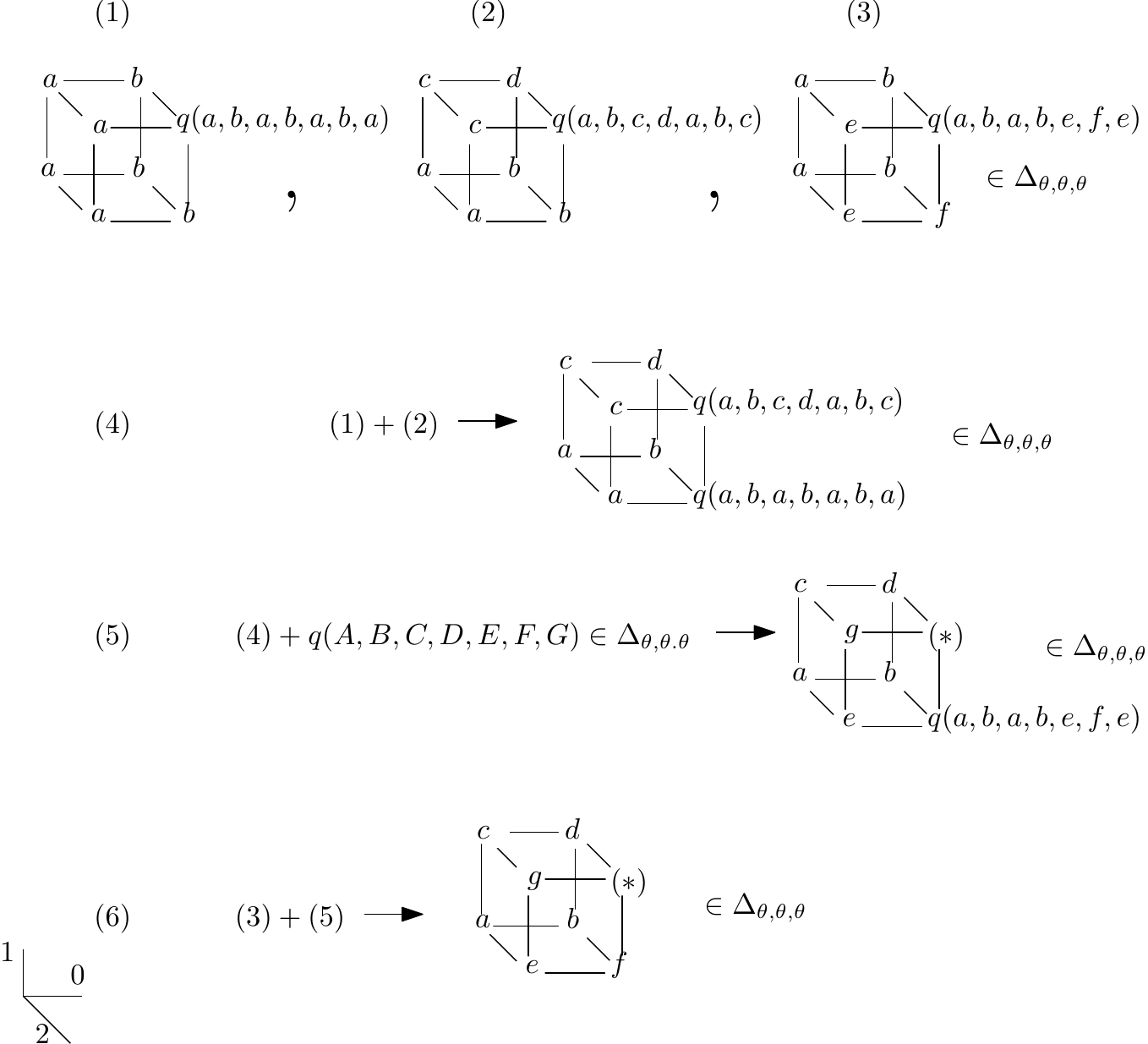}
\caption{}\label{fig:wobblycube_3}
\end{figure}

\begin{figure}
\includegraphics[scale =.9]{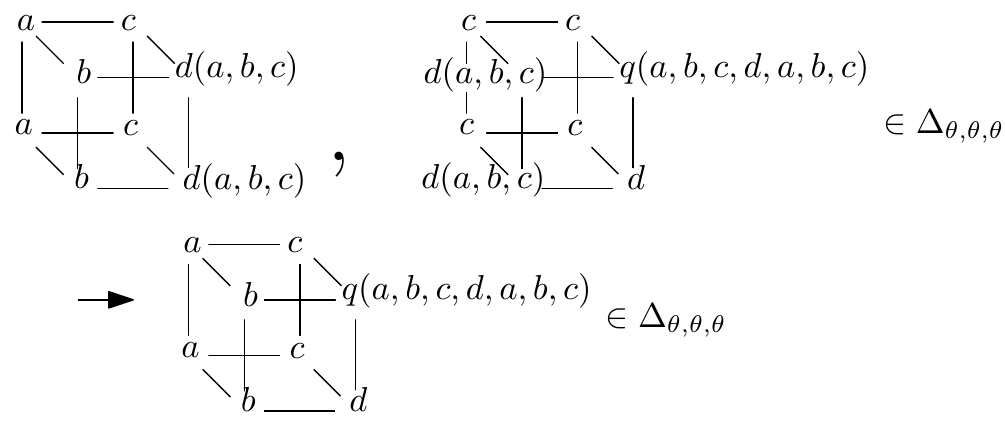}
\caption{}\label{fig:wobblycube_4}
\end{figure}
\clearpage

\begin{prop}\label{prop:wobblyidentities}
Let $q$ be a $3$-dimensional wobbly cube term for a modular variety $\var$. Then $q$ satisfies

\begin{enumerate}
\item $q(x,x,x,x,y,y,y) \approx y$
\item $q(x,x,y,y,x,x,y) \approx y$
\item $q(x,y,x,y,x,y,x) \equiv_{[\theta, \theta, \theta]} y$, where $\theta = \Cg(\langle x,y \rangle )$

\end{enumerate}

\end{prop}
\begin{proof}
The first two identities follow from the definition of $q$. Item (3) follows from Theorems \ref{thm:wobblycompletesmatrices} and \ref{thm:maindeltathm}. Let $\A \in \var$, and take $x, y \in \A$. Let $\theta = \Cg^{\A}(\langle x,y \rangle)$. Then $\cube_0^3(x,y) \in \Delta_{\theta, \theta, \theta}$. Therefore, 
$$\bigg \langle \left[ \begin{array}{cc}
					x& y\\
					x & y \end{array} \right] , \left[ \begin{array}{cc}
					x& q(x,y,x,y,x,y,x)\\
					x & y \end{array} \right] \bigg\rangle\in \Delta_{\theta, \theta, \theta},$$
and the result follows.
\end{proof}

\begin{figure}[!ht]
\includegraphics[scale=.9]{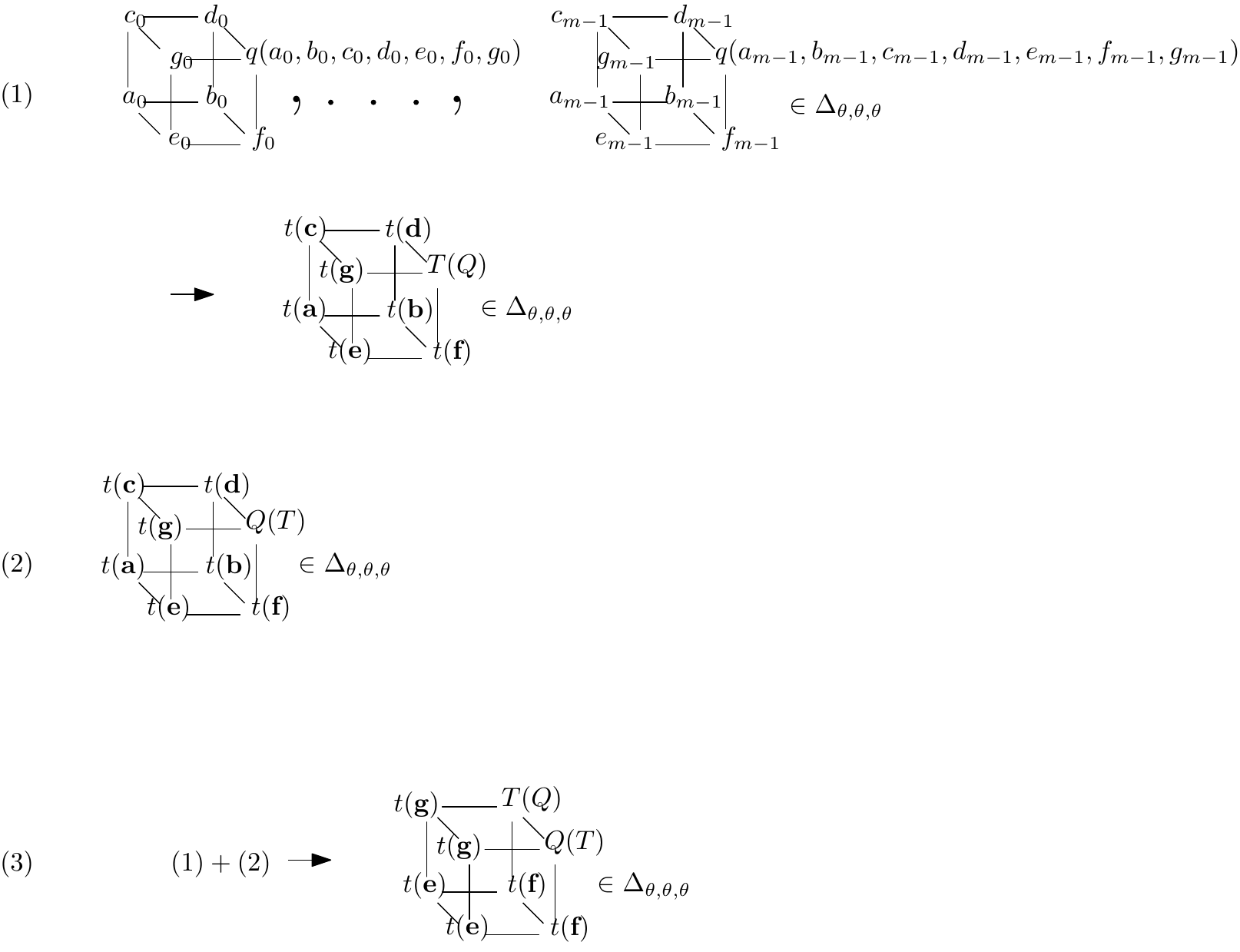}
\caption{}\label{fig:wobblycube_5}
\end{figure}

\clearpage
\bibliographystyle{plain}

\end{document}